\documentclass[12pt,reqno]{amsart}   	
\usepackage{bm} 
\usepackage{amsfonts}
\usepackage{mathrsfs}
\usepackage{yfonts}
\usepackage{url}

\usepackage[osf]{mathpazo}  

\usepackage{mathtools}

\usepackage[left=2.5cm, top=3cm, bottom=3cm, right=2.5cm]{geometry}
\usepackage{amssymb}  		
\usepackage{graphicx}
\usepackage{verbatim}
\usepackage{enumerate}	
\usepackage[english]{babel}
\usepackage{amsmath}
\usepackage{quoting}
\usepackage{array,booktabs}

\usepackage[shortlabels]{enumitem}
\usepackage{units}
\usepackage{xspace}\xspaceaddexceptions{\,}
\usepackage{multicol}
   \usepackage{tikz-cd}
\quotingsetup{font=small}	
\usepackage{amsthm}								

\usepackage[pdftex,bookmarks,bookmarksnumbered,linktocpage,   
         colorlinks,linkcolor=blue,citecolor=blue]{hyperref}

\theoremstyle{definition}
\newtheorem{definition}{Definition}
\newtheorem{example}[definition]{Example}

\newtheorem{remark}[definition]{Remark}

\theoremstyle{plain}
\newtheorem{lemma}[definition]{Lemma}
\newtheorem{proposition}[definition]{Proposition}
\newtheorem{theorem}[definition]{Theorem}
\newtheorem{corollary}[definition]{Corollary}

\newcommand\A{{\mathbf A}}
\newcommand\B{{\mathbf B}}
\newcommand\C{{\mathbf C}}

\newcommand\I{{\mathbf I}}

\newcommand\bool{\A_{\infty}}

\newcommand\PL{{\mathcal{P}}_{\l}}
\newcommand\PLA{{\mathcal{P}_{\l } (\mathbb{A})}}

\newcommand\BA{\ensuremath{\mathcal{BA}}\xspace}

\newcommand\IBSL{\ensuremath{\mathcal{IBSL}}\xspace}

\newcommand\N{\mathcal{NGIB}}
\newcommand\inj{\mathcal{IGIB}}

\newcommand\di{{\widehat{d}}}

\newcommand\eq{\thickapprox}

\newcommand{\Var}{\mathnormal{V\mkern-.8\thinmuskip ar}}

\makeatletter
\@namedef{subjclassname@2020}{%
  \textup{2020} Mathematics Subject Classification}
\makeatother

\title[]{Probability over P\l onka sums of Boolean algebras: states, metrics and topology}

\author{Stefano Bonzio}
\address{Stefano Bonzio, Artificial Intelligence Research Institute \\
Spanish National Research Council, Spain.}
\email{stefano@iiia.csic.es}
\author{Andrea Loi}
\address{Andrea Loi, Department of Mathematics and Computer Science \\
         University of Cagliari, Italy.}
         \email{loi@unica.it}

\date{}


\begin{document}

\maketitle

\begin{abstract}
The paper introduces the notion of state for involutive bisemilattices, a variety which plays the role of algebraic counterpart of weak Kleene logics and whose elements are represented as P\l onka sums of Boolean algebras. We investigate the relations between states over an involutive bisemilattice and probability measures over the (Boolean) algebras in the P\l onka sum representation and, the direct limit of these algebras. Moreover, we study the metric completion of involutive bisemilattices, as pseudometric spaces, and the topology induced by the pseudometric. 
\end{abstract}

\section{Introduction}

Probability theory is grounded on the notion of \emph{event}. Events are traditionally interpreted as elements of a ($\sigma$-complete) Boolean algebra. Intrinsically, this means that classical propositional logic is the most suitable formal language to speak about events. The direct consequence of this standard assumption is that any event can either happen (to be the case) or not happen, and, thus, its negation is taking place. One could claim that this criterion does not encompass all situations: certain events might not be either true or false (simply happen or not happen). Think about the coin toss to decide which one, among two tennis players, is choosing whether to serve or respond at the beginning of a match. Although being statistically extremely rare, the coin may fall on the edge, instead of on one face. Pragmatically, the issue is solved re-tossing the coin (hoping to have it landing on one face). Theoretically, one should admit that there are circumstances in which the event ``head'' (and so also its logical negation ``tail'') could be indeterminate (or, undefined). Interestingly enough, not all kinds of events are well modeled by classical logic. This is the case, for instance, of assertions about the properties of quantum systems (which motivated Von Neumann to introduce quantum logic \cite{VonNeumann}, see also \cite{Giuntinibook}), or assertions that can be neither true nor false (like paradoxes), or also propositions regarding vague, or fuzzy properties (like ``begin tall'' or ``being smart'' ).
Yet, we are convinced that adopting (some) non-classical logical formalism to describe certain situations is not a good objection to renounce to measure their probability. On the contrary, we endorse the idea of those who think that it constitutes a good reason to look beyond classical probability, namely to render probability when events under consideration do not belong to classical propositional logic. 


The idea of studying probability maps over algebraic structures connected to non-classical logic formalisms is at the heart of the theory of \emph{states}. A theory that is well developed for different structures involved in the study of fuzzy logics, including MV-algebras \cite{Mundici95, FlaminioKroupa}), G{\"o}del-Dummett \cite{statiGodel}, G{\"o}del$_{\Delta}$ \cite{Valotastates}, the logic of nilpotent minimum \cite{statenilpotentminimum} and product logic \cite{Flaminio2018}. Within the same strand of research, probability maps have been defined and studied also for other algebraic structures (connected to logic), such as Heyting algebras \cite{weatherson2003}, De Morgan algebras \cite{statiKleene}, orthomodular lattices \cite{statiOML} and effect algebras \cite{statieffect}. 

The idea motivating the present work is to further extend the theory of states to non-classical events; in particular, to one of the three-valued logics in the weak Kleene family, whose algebraic semantics is played by the variety of involutive bisemilattices (see \cite{Bonzio16}). The peculiarity of such variety is that each of its members has a  representation in terms of P\l onka sums of Boolean algebras. This abstract construction, originally introduced in universal algebra by J. P\l onka \cite{Plo67,Plo67a}, is performed over direct systems of algebras whose index set is a semilattice. The axiomatisation of states we propose, which is motivated by the logic PWK (Paraconsistent Weak Kleene), allows to ``break'' a state into a family of (finitely additive) probability measures over the Boolean algebras in the P\l onka sum representation of an involutive bisemilattice. In other words, our notion of state accounts for (and is strictly connected to) all the Boolean algebras in the P\l onka sum. 
Moreover, we show that states over an involutive bisemilattice are in bijective correspondence with finitely additive probability measures over the Boolean algebra constructed as the direct limit of the algebras in the (semilattice) direct system of the representation. This allows to prove that each state corresponds to an integral over the dual space of the direct limit (the inverse limit of the dual spaces). 

The results obtained explore, on the one hand, the possibility of defining probability measures over the (non-equivalent) algebraic semantics of certain Kleene logics. On the other, it shows how (finitely additive) probability measures can be lifted from Boolean algebras to the P\l onka sum of Boolean algebras.

The paper is organised as follows. Section \ref{Sec: Preliminari} recaps all the necessary preliminary notions helpful for the reader to go through the entire paper. In Section \ref{Sec: states}, states over involutive bisemilattices are introduced. We show that each (non-trivial) involutive bisemilattice, having no trivial algebra in its P\l onka sum representation, carries at least a state. This class has been introduced (see \cite{Paoliextensions}) for its logical relevance, as algebraic counterpart of an extension of PWK. Moreover, we show the correspondence with probability measures of the direct limit of the Boolean algebras in the semilattice system of $\B$. We dedicate Section \ref{Sec: faithful states} to the analysis of strictly positive states, which we refer to as \emph{faithful states} (in accordance with the nomenclature for MV-algebras). In particular, the presence of a faithful state motivates the introduction of the subclass of \emph{injective} involutive bisemilattices, characterised by injective homomorphisms in the P\l onka sum representation. In Section \ref{Sec: metric and completition} and \ref{Sec: topology}, respectively, we approach involutive bisemilattices carrying a state as pseudometric spaces and topological spaces (with the topology induced by the pseudometric), respectively. In the former we study metric completions, while in the latter we insist on the relation between involutive bisemilattices and the correspondent direct limits, as topological spaces. We close the paper with Section \ref{sec: conclusioni}, introducing possible further works and two Appendixes, discussing, respectively, an alternative notion of state (and motivating why we discard it) and some details of an unsolved problem stated in Section \ref{Sec: faithful states}.

\section{Preliminaries}\label{Sec: Preliminari}

\subsection{P\l onka sums}

A \textit{semilattice} is an algebra $\I = \langle I, \lor\rangle$ of type $\langle 2\rangle$, where $\lor$ is a binary commutative, associative and idempotent operation. Given a semilattice $\I$, it is possible to define a partial order relation between the elements of its universe, as follows
\[
i \leq j \Longleftrightarrow i \lor j = j,
\]
for each $i,j\in I$. We say that a semilattice has a \emph{least element}, if there exists an element $i_0\in I$ such that $i_{0}\leq i$ (equivalently, $i_0\vee i = i$), for all $i\in I$.

\begin{definition}\label{Def:Direct Systems of algebras}
A \textit{semilattice direct system of algebras} is a triple $\mathbb{A}=\langle \{\A_{i}\}_{i\in I}, I, p_{ij} \rangle$ consisting of
\begin{enumerate}
\item a semilattice\footnote{With a slight abuse of notation, we identify semilattice $\mathbf{I}$ with its universe $I$.} $I = \langle I, \lor\rangle$ with least element $i_{0}$;
\item a family of algebras $\{ \A_{i}\}_{i \in I}$ of the same type with disjoint universes;
\item a homomorphism $p_{ij} \colon \A_{i} \to \A_{j}$, for every $i, j \in I$ such that $i \leq j$,
\end{enumerate}
where $\leq$ is the order induced by the binary operation $\vee$.

\noindent 
Moreover, $p_{ii}$ is the identity map for every $i \in I$, and if $i \leq j \leq k$, then $p_{ik} = p_{jk} \circ p_{ij}$.
\end{definition}

Organising a family of algebras $\{\A_i\}_{i\in I}$ into a semilattice direct system means, substantially, requiring that the index set $I$ forms a semilattice and that algebras whose indexes are comparable with respect to the order are ``connected'' by homomorphisms, whose ``direction'' is bottom up, namely from algebras whose index is lower to algebras with a greater index. 
The nomenclature in Definition \ref{Def:Direct Systems of algebras} is deliberately chosen to emphasise the presence of an index set equipped with the structure of \emph{semilattice}\footnote{Systems of this kind are special cases of \emph{direct systems} (of algebras), which differentiate with respect to the index set which is assumed to be a directed preorder.}. We will often refer to a semilattice direct system simply as $\{\A_{i}\}_{i\in I} $ (instead of $\mathbb{A}=\langle \{\A_{i}\}_{i\in I}, I, p_{ij} \rangle$) and, in order to indicate homomorphisms, we sometimes write $p_{i,j}$ instead of $p_{ij}$ (in situations where confusion may arise).

The P\l onka sum is a new algebra that is defined given a semilattice direct systems of algebras. 

\begin{definition}\label{def: Plonka sum of algebras}
Let $\mathbb{A}=\langle \{\A_{i}\}_{i\in I}, I, p_{ij} \rangle$ be a semilattice direct system of algebras of type $\tau$. The \textit{P\l onka sum} over $\mathbb{A}$, in symbols $\PL(\mathbb{A})$ or $\PL(\A_{i})_{i \in I}$, is the algebra such that
\begin{enumerate}
\item the universe of $\PL(\mathbb{A})$ is the disjoint union $\displaystyle{\bigsqcup_{i \in I}}A_{i}$;

\item for every $n$-ary basic operation $f$ (with $n \geqslant 1 $) in $\tau$, and $a_{1}, \dots, a_{n} \in \bigcup_{i \in I}A_{i}$, we set
\[
f^{\PL(\A_{i})_{i \in I}}(a_{1}, \dots, a_{n}) \coloneqq f^{\A_{j}}(p_{i_{1} j}(a_{1}), \dots, p_{i_{n} j}(a_{n}))
\]
where $a_{1} \in A_{i_{1}}, \dots, a_{n} \in A_{i_{n}}$ and $j = i_{1} \lor \dots \lor i_{n}$;\ 

for every $0$-ary operation $c$ (constant) in $\tau$, 
\[
c^{\PL(\A_{i})_{i \in I}} \coloneqq c^{\A_{i_0}}, 
\]
where $i_{0}$ is the least element in $I$.

\end{enumerate}	 
\end{definition}


In words, non-nullary operations $f^{\PLA}$ on the elements $a_1,\dots, a_n$ of the P\l onka sum are defined by computing the operation $f$ in the algebra $\A_j$, whose index is the join of the indexes (a notion that is well defined since the index set $I$ is a semilattice) corresponding to the algebras where the elements $a_1,\dots,a_n$ live, respectively (this idea is clarified through Example \ref{ex: operazioni Plonka}). The constants of the P\l onka sum coincide with the constants of the algebra $\A_{i_0}$ (whose index set is the least element in $I$).

The theory of P\l onka sums is intrinsically connected with the notion of partition function which we recall in the following. 

\begin{definition}\label{def: partition function}
Let $\A$ be an algebra of type $\tau$. A function $\cdot\colon A^2\to A$ is a \emph{partition function} in $\A$ if the following conditions are satisfied for all $a,b,c\in A$, $ a_1 , ..., a_n\in A^{n} $ and for any operation $g\in\tau$ of arity $n\geqslant 1$, and $c\in\tau$ of arity $n=0$.
\begin{enumerate}
\item[(PF1)] $a\cdot a = a$,
\item[(PF2)] $a\cdot (b\cdot c) = (a\cdot b) \cdot c $,
\item[(PF3)] $a\cdot (b\cdot c) = a\cdot (c\cdot b)$,
\item[(PF4)] $g(a_1,\dots,a_n)\cdot b = g(a_1\cdot b,\dots, a_n\cdot b)$,
\item[(PF5)] $b\cdot g(a_1,\dots,a_n) = b\cdot a_{1}\cdot_{\dots}\cdot a_n $, 
\item[(PF6)] $a\cdot c = a$.
\end{enumerate}
\end{definition} 

A comment on the above definition is in order. The upshot of (PF1)-(PF3) is that $\langle A,\cdot\rangle$ is a \emph{left normal band}. Left normal bands (see e.g. \cite[Ch. 2]{Petrich} or \cite[Sec. 4.4-4.6]{Howie}) are important classes of algebras in semigroup theory. (PF4)-(PF6) demand some compatibility between $\cdot$ and the operations in the type $\tau$ (which may include constants). It is useful to recall that the above definition uses the minimal number of conditions, in case the type $\tau$ admits also constants\footnote{The construction of the P\l onka sum and the notion of partition function were originally introduced by P\l onka in \cite{Plo67a,Plo67} considering only types without constants and, subsequently, extended to types with constants \cite{plonka1984sum}.}; it is not difficult to check that the above introduced definition is equivalent to the one in \cite{plonka1984sum} (see also \cite[Section 6]{Romanowska92}).

The connection between partition functions and P\l onka sums is provided by the following result, which states that algebras possessing a (term-definable) partition function can be decomposed as P\l onka sums.

\begin{theorem}\cite[Thm.~II]{plonka1984sum}\label{th: Teorema di Plonka}
Let $\A$ be an algebra of type $\tau$ with a partition function $\cdot$. The following conditions hold: 
\begin{enumerate}
\item $A$ can be partitioned into $\{ A_{i} \}_{i \in I }$ where any two elements $a, b \in A$ belong to the same component $A_{i}$ exactly when
\[
a= a\cdot b \text{ and }b = b\cdot a.
\]
\item The relation $\leq$ on $I$ given by the rule
\[
i \leq j \Longleftrightarrow \text{ there exist }a \in A_{i}, b \in A_{j} \text{ s.t. } b\cdot a =b
\]
is a partial order and $\langle I, \leq \rangle$ is a semilattice with least element.\footnote{With a slight abuse of notation we indicate here the semilattice of indexes $I$ by the order $\leq$ and not by the binary operation $\vee$.}
\item For all $i,j\in I$ such that $i\leq j$ and $b \in A_{j}$, the map $p_{ij} \colon \A_{i}\to \A_{j}$, defined by the rule $p_{ij}(x)= x\cdot b$ is a homomorphism. 
\item $\mathbb{A} = \langle \{ \A_{i} \}_{i \in I}, \langle I, \leq \rangle, \{ p_{ij} \! : \! i \leq j \}\rangle$ is a direct system of algebras such that $\PL(\mathbb{A})=\A$.
\end{enumerate}
\end{theorem}
Theorem \ref{th: Teorema di Plonka} is enunciated in the general form which allows the presence of constants in the type $\tau$. 
The following example may be useful for the reader to understand how the construction of P\l onka sum is performed over a semilattice direct system of Boolean algebras (and the relation with the partition function).

\begin{example}\label{ex: operazioni Plonka}
Consider the four-elements semilattice $I=\{i_{0},i,j,k\}$ whose order is given as follows:
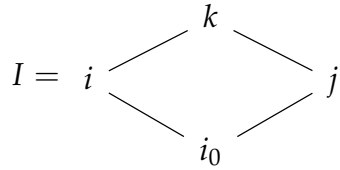
\[
I = \begin{tikzcd}[row sep = tiny, arrows = {dash}]
 & k & \\ i \arrow[ur, dash] & & j \arrow[ul] \\ & i_{0}\arrow[ul]\arrow[ur] &
  \end{tikzcd}
  \]

Consider the family $\{\A_{i_{0}}, \A_{i}, \A_{j}, \A_{k}\}$ of Boolean algebras, organised into a semilattice direct system, whose index is $I$ and homomorphisms are defined as extensions of the following maps: $p_{ik}(a) = c$, $p_{jk}(b) = e$ (thus, $p_{ik}(a') = c'$, $p_{jk}(b') = e'$ ), $p_{i_{0}m}$ is defined in the unique obvious way, for any $m\in\{i,j,k\}$.

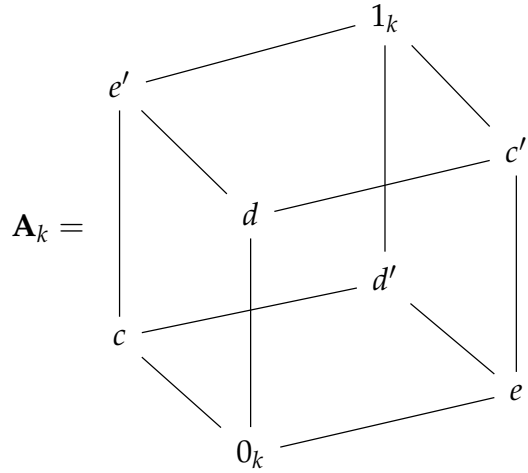
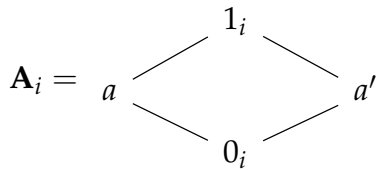
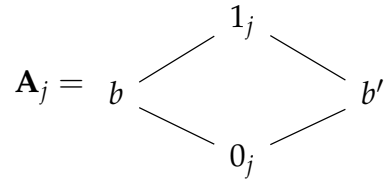
\[
\A_{k} = \begin{tikzcd}[row sep = tiny, arrows = {dash}]
& & 1_{k}& \\
e'\arrow[urr] &&& \\
 &&&c'\arrow[uul] \\
& d\arrow[urr]\arrow[uul] & &  \\
& & d'\arrow[uuuu] & \\
 c\arrow[uuuu]\arrow[urr] & & &  \\ 
 & &  & e\arrow[uuuu]\arrow[uul] \\ 
 & 0_{k}\arrow[uuuu] \arrow[uul]\arrow[urr]& &
  \end{tikzcd}
\]

\[
\A_i = \begin{tikzcd}[row sep = tiny, arrows = {dash}]
 & 1_i & \\ a \arrow[ur, dash] & & a' \arrow[ul] \\ & 0_{i}\arrow[ul]\arrow[ur] &
  \end{tikzcd}
  \qquad  \qquad
  \A_j = \begin{tikzcd}[row sep = tiny, arrows = {dash}]
 & 1_j & \\ b \arrow[ur, dash] & & b' \arrow[ul] \\ & 0_j\arrow[ul]\arrow[ur] &
  \end{tikzcd}
\]
\[
  \A_{i_{0}} = \begin{tikzcd}[row sep = tiny, arrows = {dash}]
 & 1 & \\
   & &  \\
    & 0\arrow[uu] &
  \end{tikzcd}
\]

According to Definition \ref{def: Plonka sum of algebras}, the P\l onka sum over the above introduced semilattice direct system is the new algebra $\B$, whose universe is $B=A_{i_{0}}\sqcup A_{i}\sqcup A_{j}\sqcup A_{k}$. The constants in $\B$ are the constants of the Boolean algebra $\A_{i_{0}}$.
We just give an example of how binary operations are computed in $\B$ (as should be clear that $'$ coincides with negation in each respective algebra). 

\[
a \wedge^{\B} a' = p_{ii}(a) \wedge^{\A_{i}} p_{ii} (a') = a \wedge^{\A_{i}} a' = 0_{i}. \]
In words, a binary operation between elements belonging to the same algebra (e.g. $\A_{i}$) is computed as in that algebra. As for binary operations between elements ``living'' in (the universes of) two different algebras (e.g. $\A_{i},\A_{j}$) it is necessary to recur to the algebra $\A_{k}$ (since $k= i\vee j$) and homomorphisms in the system:
\[
a \vee^{\B} b = p_{ik}(a) \vee^{\A_{k}} p_{jk} (b) = c \vee^{\A_{k}} e = d'. \]

The algebra $\B$ just introduced is an example of involutive bisemilattice (see Definition \ref{def: involutive bisemilattices}) and may be helpful to clarify also the notion of partition function and the content of Theorem \ref{th: Teorema di Plonka}.\footnote{We thank an anonymous reviewer for urging the need for a clarification.} Observe that $\B$ does not satisfy absorption. Indeed, $a\wedge^{\B}(a\vee^{\B} b) = p_{ik}(a)\wedge^{\A_{k}}(p_{ik}(a)\vee^{\A_{k}}p_{jk}(b)) = c\neq a$. However, all the algebras $\{\A_{i_{0}}, \A_{i}, \A_{j}, \A_{k}\}$ do satisfy absorption (they are Boolean algebras). It is indeed easy to verify that the term function $x\cdot y\coloneqq x\wedge(x\vee y)$ is a partition function on $\B$ (it does satisfy all the properties in Definition \ref{def: partition function}). In order to clarify the content of Theorem \ref{th: Teorema di Plonka}, consider $\B=\langle B, \vee,\wedge,^{'},0,1\rangle$ as an algebra of type $(2,2,1,0,0)$, and the term operation $x\cdot y\coloneqq x\wedge(x\vee y)$ playing the role of partition function. Then $\B$ can be partitioned into (four) components each of which satisfies $x\cdot y \approx x$, namely absorption: these components clearly are $\A_{i_{0}}, \A_{i}, \A_{j}, \A_{k}$. Moreover, the set of indexes $\{i_{0},i,j,k\}$ can be equipped with the structure of semilattice, according to condition (2) in Theorem \ref{th: Teorema di Plonka}. For instance, $i_{0}\leq i$ as $1_{i}\cdot 1 = 1_{i}\wedge^{\B}(1_{i}\vee^{\B}1) = 1_{i}$; $i\leq k$ since $1_{k}\cdot 1_{i} = 1_{k}\wedge^{\B}(1_{k}\vee^{\B}1_{i}) = 1_{k}$. Condition (3) guarantees the presence of homomorphisms between the (Boolean) algebras $\{\A_{i_{0}}, \A_{i}, \A_{j}, \A_{k}\}$. For instance, $p_{ik} \colon \A_{i}\to \A_{k}$ is defined by the rule $p_{ik}(x)= x\cdot y$ (for some $y\in A_{k}$). For instance, $p_{ik}(a) =a\cdot 1_{k} = c\wedge(c\vee 1_{k}) = c $. It is not difficult to check that the definition of $p_{ik}$ does not depend on the choice of a specific element in the algebra $\A_{k}$.  

\end{example}

\subsection{Involutive Bisemilattices}

\begin{definition}\label{def: involutive bisemilattices}
An \emph{involutive bisemilattice} is an algebra $\B = \langle B,\land,\lor,^{'},0,1\rangle$ of type $(2,2,1,0,0)$ satisfying:
\begin{enumerate}[label=\textbf{I\arabic*}.]
\item $x\lor x\eq x$;
\item $x\lor y\eq y\lor x$;
\item $x\lor(y\lor z)\eq(x\lor y)\lor z$;
\item $ (x')'\eq x$;
\item $x\land y\eq( x'\lor y')'$;
\item $x\land(x'\lor y)\eq x\land y$; \label{rmp}
\item $0\lor x\eq x$;
\item $1\eq 0'$.
\end{enumerate}

\end{definition}

It follows from the definition, that the $\vee$-reduct of an involutive bisemilattice is a (join) semilattice and, consequently, the binary operation $\vee$ induces a partial order $\leq_{\vee}$. Similarly, in virtue of \textbf{I5}, also the $\wedge-$reduct is a (meet) semilattice, inducing a partial order $\leq_{\wedge}$. In general, the two orders are different\footnote{It is easy to check that $x\leq_{\vee} y$ if and only if $y'\leq_{\wedge}x'$ (see \cite{Bonzio16}).} and are not lattice-orders: they are (and actually coincide) in case an involutive bisemilattice is a Boolean algebra.
The class of involutive bisemilattices forms a variety which we denote by $\IBSL$, which satisfies all (and only) the regular identities\footnote{An identity $ \varphi \approx \psi $ is \emph{regular} provided that $\Var(\varphi) = \Var(\psi)$.} holding for Boolean algebras. Thus, for instance, it fails to satisfy absorption (see \cite{Bonzio16} for details). The importance of involutive bisemilattice is connected to logic, as (one of its subquasivarieties) plays the role of algebraic semantics of Paraconsistent Weak Kleene logic (see \cite{Bonzio16}).

Involutive bisemilattices are connected to P\l onka sums. In particular, every involutive bisemilattice is the P\l onka sum over a semilattice direct system of Boolean algebras and, conversely, the P\l onka sum over any semilattice direct system of Boolean algebras is an involutive bisemilattice (see \cite{Bonzio16, Romanowska92}). It follows that the P\l onka sum introduced in Example \ref{ex: operazioni Plonka} is an example of involutive bisemilattice. Some details of the representation theorem are exemplified in the following (for further details, we refer the reader to \cite{Bonzio16}).

\begin{example}\label{ex: nuovo esempio}
Let $\B\in\IBSL$, with $B=\{0,1,a,a',b,b'\}$, whose order $\leq_{\vee} $ is represented in the following Hasse diagram.

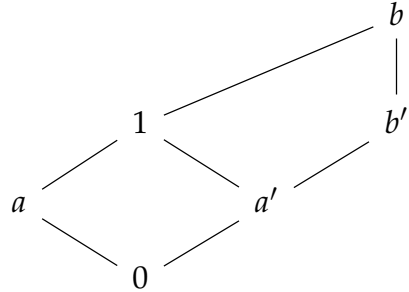
\[
\begin{tikzcd}[row sep = tiny, arrows = {dash}]
& & & b \\
&&& \\
& & & \\
& & & \\ 
 & 1\arrow[uuuurr] & & b'\arrow[uuuu]   \\
 & & &  \\
  a \arrow[uur, dash] & & a' \arrow[uul]\arrow[uur] & \\
  & & &  \\
   & 0\arrow[uul]\arrow[uur] & & 
  \end{tikzcd}
\]

\noindent
Define $x\cdot y\coloneqq x\wedge(x\vee y)$. Observe that $\B$ does not satisfy absorption, indeed: $a\cdot b = a\wedge (a\vee b) = a\wedge b = (a'\vee b')' = (b')' = b $. It is not difficult to check that $\cdot $ is a partition function on $\B$. By Theorem \ref{th: Teorema di Plonka}-(1), $B$ can be partitioned in two components: $A_{i_{0}} = \{0,a,a',1\}$, $A_{j} = \{b,b'\}$. Indeed (just to show one case), $a\cdot a' = a\wedge (a\vee a') = a\wedge 1 = (a'\vee 0)' = (a')' = a$ (and, similarly, $a'\cdot a = a'$). Condition (2) in Theorem \ref{th: Teorema di Plonka} implies that the semilattice of indexes is indeed the two-elements chain $\{i_{0}, j\}$ with $i_{0} <\ j$. Indeed, for instance, $b\cdot a = b\wedge (b\vee a) = b\wedge b = b$. Condition (3) fixes the (unique) homomorphism $p_{i_{0}j}\colon \A_{i_{0}}\to\A_{j}$. To exemplify just one case: $p_{i_{0}j}(a) = a\cdot b = b$ (it is immediate to check that $p_{i_{0}j}(1) = b$, $p_{i_{0}j}(0) = p_{i_{0}j}(a') = b'$). The elements $b$ and $b'$ are the top and bottom element, respectively, of the Boolean algebra $\A_{j}$. This illustrates the P\l onka sum representation of $\B$.

\end{example}

\begin{remark}
Throughout the whole paper, when considering $\B\in\IBSL$ we will always identify it with its P\l onka sum representation $\PL(\A_{i})$, without explicitly mentioning it.
\end{remark}

\noindent
\textbf{Notation:} we denote by $1_{i}$, $0_i$ the top and bottom element, respectively, of the Boolean algebras $\A_i $ (in the P\l onka sum representation of $\B\in\IBSL$).

\subsection{Finitely additive probability measures over Boolean algebras}

Let $\A$ be a Boolean algebra. A finitely additive probability measure over $\A$ is a real-valued map $m\colon\A\to [0,1]$ such that: 
\begin{enumerate}
\item $m(1) = 1$; 
\item $m(a\vee b ) = m(a) + m(b)$, if $a\wedge b = 0$. 
\end{enumerate}

With a slight abuse of notation here, and elsewhere, we use the same symbol (``$1$'') to denote both the top element of a Boolean algebra and the unit of $\mathbb{R}$. A probability measure $m$ over a Boolean algebra $\A$ is called \emph{regular} (or, strictly positive) provided that $m(a) > 0$ for any $a\neq 0$.
\\
\noindent
For our purposes, it is useful to recall a relevant result which connects probability maps over Boolean algebras to probability measures over an appropriate topological space. This result (proved, independently, by Kroupa \cite{Kroupa2006} and Panti \cite{Panti}), holds, more in general, for MV-algebras.\footnote{We opted not to recall basic facts about MV-algebras as they are unnecessary to go through the paper. Standard references are \cite{Mundicibook} or \cite{DiNola}.} For the convenience of our reader, we opt to formulate it for Boolean algebras, recalling that every Boolean algebra is a semisimple MV-algebra (where the operations $\oplus$ and $\odot$ coincide with $ \vee$ and $ \wedge$, respectively). The fact that a Boolean algebra $\A$ is semisimple, as MV-algebra, allows to represent it as an algebra of $[0,1]$-valued continuous functions defined over its (dual) Stone space $\A^{\ast}$ (see \cite[Theorem 2.1.7]{FlaminioKroupa}). It follows that one can associate to each element $a\in A$, an unique continuous function $a^{\ast}\colon\A^{\ast}\to [0,1]$. 
Moreover, we refer to the space of all the (finitely additive) probability measures over a Boolean algebra $\A$ as $\mathcal{S}(\A)$. 

\begin{theorem}\label{th: rappresentazione integrale misure di probabilita}
Let $\A$ be a Boolean algebra and $m\colon\A\to [0,1]$ a (finitely additive) probability measure. Then
\begin{enumerate}
\item There is a homeomorphism $\Psi\colon \mathcal{S}(\A)\to\mathcal{M}(\A^{\ast})$, where $\mathcal{M}(\A^{\ast})$ is the space of all regular\footnote{Notice that this use of the term ``regular'' is different from above and refers to Borel measures over topological spaces (see, for instance, \cite{FlaminioKroupa}).} Borel probability measures on its dual space $\A^{\ast}$.
\item For every $a\in A $, 
\[ 
m(a) = \int_{\A^{\ast}} a^{\ast}(M) \; d\mu_{s}(M), 
\]
where $a^{\ast}$ is the unique function associated to $a$, $M\in \A^{\ast}$ and $d\mu_{s}= \Psi(s) $. 
\end{enumerate}
\end{theorem}

A standard reference for the above theorem is \cite[Theorem 4.0.1]{FlaminioKroupa} (where it is stated and proved for MV-algebras).

\subsection{Booleanisation of an involutive bisemilattice}\label{subsec: Booleanisation}

Given a semilattice direct system of algebras, the P\l onka sum is only one way to construct a new algebra. Another is the \emph{direct limit}. We recall this construction in the special case of Boolean algebras (we refer the reader interested into the details to \cite{haimo}).

Consider an involutive bisemilattice $\B\cong\PL(\A_{i})$. The \emph{direct limit} over a direct system $\{\A_{i}\}_{i\in I}$ of Boolean algebras is the Boolean algebra defined as the quotient: 
\[
\lim_{\to_{i\in I}}\A_i\coloneqq(\bigsqcup_{i\in I}\A_{i})/\sim,
\]

where, for $a\in A_{i}$ and $b\in A_{j}$, $a\sim b$ if and only if there exist $c\in A_{k}$, for some $k\in I$ with $i,j\leq k$, such that $p_{ik}(a) = c = p_{jk}(b)$. It is immediate to check that $\sim$ is a congruence on the (disjoint) union, thus, operations on the direct limit are defined as usual for quotient algebras.

 It is always possible to associate to $\B\in\IBSL$, the Boolean algebra $\displaystyle\lim_{\to}\A_i$, the direct limit of the algebras (in the system) $\{\A_i\}_{i\in I}$, which we will call the \emph{Booleanisation}\footnote{It is useful to recall that the present use of the term ``Booleanisation'' differs from other usages in literature, in lattice theory (see \cite{BORLIDO}) and in the theory of Boolean (inverse) semigroups (see \cite{Booleansemigroups}).} of $\B$, and we will indicate it by $\A_{\infty}$. 
Given an involutive bisemilattice $\B$, we can define the map $\pi\colon\B\to \A_{\infty}$, as $ \pi(a)\coloneqq [a]_{\sim} $, for any $a\in B$.

\begin{remark}
The map $\pi$ is a surjective homomorphism. We provide the details of one binary operation only (the meet $\land$). Suppose $a,b\in B$ with $a\in A_{i}$, $b\in A_{j}$ and $k=i\vee j$, then: 
\[
\pi(a\wedge b)= [a\wedge b]_{\sim}=[p_{ik}(a)\wedge p_{jk}(b)]_{\sim}=[\;p_{ik}(a)]_{\sim}\wedge [\;p_{jk}(b)]_{\sim} = \pi(a)\wedge\pi (b), 
\]
where the second last equality is justified by observing that $c\in [p_{ik}(a)\wedge p_{jk}(b)]_{\sim} $ (for an arbitrary $c\in A_l$ and $m=k\vee l$) if and only if $p_{lm}(c)= p_{km}(p_{ik}(a)\wedge p_{jk}(b)) = p_{km}(p_{ik}(a))\wedge p_{km}(p_{jk}(b)) = p_{im}(a)\wedge p_{jm}(b)$ and this is equivalent to say $c\in [\;p_{ik}(a)]_{\sim}\wedge [\;p_{jk}(b)]_{\sim} $.
\end{remark}

\noindent
\textbf{Notation:} to indicate elements of the Booleanisation we sometimes drop the subscript $_{\sim}$ when no danger of confusion may arise.

\section{States over Involutive Bisemilattices}\label{Sec: states}


\begin{definition}\label{def: stato su IBSL}
Let $\B$ be an involutive bisemilattice. A \emph{state} over $\B$ is a map $s\colon\B\to[0,1]$ such that:
\begin{enumerate}
\item $s(1) = 1$; 
\item $s(a\lor b)= s(a)+s(b)$, provided that $a\land b \in \bigcup_{i\in I}\{0_i\}$.
\end{enumerate}
\end{definition}

Moreover, a state $s\colon\B\to[0,1]$ is \emph{faithful} if $s(a)> 0$, for every $a\neq 0$.

The following resumes the basic properties of a state over an involutive bisemilattice.

\begin{proposition}\label{prop: basic properties}
Let $s$ be a state over an involutive bisemilattice $\B$. Then
\begin{enumerate}
\item $s(0)=0$;
\item $s(1_i)=1 $ and $s(0_i)=0$, for every $i\in I$. 
\item $s(a') = 1- s(a)$, for every $a\in B$.
\end{enumerate}

\end{proposition}

\begin{proof}
(1) Since $0\land 1 = 0\in \bigcup_{i\in I}\{0_{i}\}$, then $s(1)=s(0\vee 1)= s(0)+ s(1)$, hence $s(0)=0$.\\
\noindent
(2) Observe that $0_i\wedge 1 = 0_i\wedge p_{i_{0}i}(1)= 0_i \wedge 1_i = 0_i $. Then $s(0_i) + s (1) = s(0_i\vee 1) = s(0_i\vee 1_i) = s(0_i) + s (1_i) $. Therefore $s(1_i) = s(1) = 1$. $s(0_i) = 0$ follows observing that $0_i \vee 1_i = 1_i$.\\
\noindent
(3) Let $a\in A_i$ for some $i\in I$. Then $a'\in A_i$ and $a\wedge a' = 0_i$. Therefore $1 = s(1_i) = s(a\vee a') = s(a) + s(a')$.
\end{proof}

The motivating idea for introducing states over involutive bisemilattices is that of having maps expressing the probability of events that are different from classical ones (corresponding to elements of a Boolean algebra). 
Our definition of states relies on the algebraic characterisation of the logic PWK given in \cite{Bonzio16}. We try to motivate the logical reason behind our choice, looking at the involutive bisemilattice introduced in Example \ref{ex: operazioni Plonka}, which will refer to as $\B$ (in this paragraph). According to that analysis carried out in \cite{Bonzio16}, the only logical filter turning matrix $\langle \B, F\rangle$ into a reduced model of PWK is $F= \{1_{i_{0}}, 1_{i}, 1_{j}, 1_{k}\}$, i.e. the union of the top elements of the Boolean algebras in the P\l onka sum representation of $\B$ (this can be deduced also from the results in \cite{BonzioMorascoPraBaldi}). In other words, $F= \{1_{i_{0}}, 1_{i}, 1_{j}, 1_{k}\}$ expresses the notion of truth with respect to the logic PWK (on the algebra $\B$). Requiring a state $s$ to map the constant $1$ (of an involutive bisemilattice) into $1$ (Condition (1) in Definition \ref{def: stato su IBSL}) is equivalent to say that true events (according to the logic PWK) have probability $1$ (this is a consequence also of the choice of Condition (2), since it follows from Proposition \ref{prop: basic properties}). Condition (2) is introduced in analogy with the classical condition according to which the probability of logically incompatible events is additive. In classical logic, two events $x$ and $y$ are incompatible when their conjunction is the impossible event ($x\wedge y= 0$, in a Boolean algebra). Shifting from classical logic to PWK implies adopting a different notion of incompatibility between events. Two different ways of rendering incompatibility appear then ``natural'' for an involutive bisemilattice: 1) to say that events $x$ and $y$ are incompatible if $x\wedge y = 0$ or; 2) that events $x$ and $y$ are incompatible if $x\wedge y = 0_{i}$, where $0_{i}$ is the bottom element of the Boolean algebra (in the P\l onka sum) where $x\wedge y$ is computed. We opted for 2) persuaded by the idea that two events in the logic PWK are \emph{not compatible} when their conjunction ($\wedge$ in an involutive bisemilattice) is not simply untrue (not belonging to a logical filter), but, more specifically, the \emph{opposite} of truth. And this is obviously obtained choosing the union of zeros in the Boolean algebra in the P\l onka sum representation. The choice of the other mentioned notion of incompatibility leads to a different scenario, briefly discussed in Appendix A.

\begin{definition}\label{def: omomorfismi che commutano con le misure}
Let $\B\in\IBSL$. For every $i,j\in I$ and a family of finitely additive probability measures $\{m_i\}_{i\in I}$ over $\{\A_i\}_{i\in I}$ (each $\A_{i}$ carries the measure $m_{i}$), the homomorphism $p_{ij}$ \emph{preserves the measures} if $m_j(p_{ij}(a)) = m_i (a)$, for any $a\in A_i$. 
\end{definition}

\noindent
\textbf{Notation:} we indicate the restriction of a state $s$ on the Boolean components $\A_{i},\A_{j}\A_{k},\dots$ of the P\l onka sum representation of an involutive bisemilattice $\B$, as $s_{i}, s_{j}, s_{k},\dots$ instead of $s_{|_{\A_{_i}}}, s_{|_{\A_{_j}}}, s_{|_{\A_{_k}}}, \dots$ to make notation less cumbersome.

\begin{proposition}\label{prop: restrizione su Ai e uno stato}
Let $s$ be a map from $\B$ to $[0,1]$. The following are equivalent:  
\begin{enumerate}
\item $s$ is a state over $\B$; 
\item $s_{i}\colon \A_{i}\to [0,1]$ is a (finitely additive) probability measure over $\A_i$, for every $i\in I$, and $p_{ij}$ preserves the measures, for each $i\leq j$.
\end{enumerate}
\end{proposition}

\begin{proof}
(1) $\Rightarrow$ (2). Assume that $s$ is a state over $\B$. Then,
$s_i (1_i) = s(1_i) = 1$ for each $i\in I$, by Proposition \ref{prop: basic properties}-(2). Moreover, let $a,b\in A_i$ be two elements such that $a\land^{\A_{i}} b = 0_i $. Observe that $a\wedge^{\B} b = a\wedge^{\A_{i}} b$ and $a\vee^{\B} b = a\vee^{\A_{i}} b$. Then $s_i (a\lor b)= s(a \lor b) = s(a) + s(b)$, using (2) in Definition \ref{def: stato su IBSL}. This shows that $s_i$ is a (finitely additive) probability measure over $\A_{i}$, for every $i\in I$. 
Now, let $a\in A_i$, for some $i\in I$, and $j\in I$ with $i\leq j$. Observe that $a\wedge^{\B} 0_j = p_{ij}(a)\wedge^{\A_{j}} 0_j = 0_j $, thus 

$$ s_j (p_{ij}(a)) = s(p_{ij}(a)) = s(p_{ij}(a)\vee^{\A_{j}} 0_j) = s(a\vee^{\B} 0_j)  = s(a) + s(0_{j}) = s(a) = s_i (a), $$
where we have used the additivity of a state and Proposition \ref{prop: basic properties}-(1). 
This shows that any homomorphism $p_{ij}$ preserves the measures $s_i$, $s_j$, for every $i\leq j $.

(2) $\Rightarrow$ (1) Assume that every Boolean algebra $\A_{i}$ in the direct system carries a finitely additive probability measure $s_i$ and that each homomorphism $p_{ij} $ ($i\leq j$) preserves the measures. Let $s\colon\B\to[0,1]$ be the map defined as 
\[
s(x)\coloneqq s_i(x),
\]
for $x\in A_i$, with $i\in I$.

\noindent
(1) $s(1) = s_{i_0}(1) = 1$.\\
\noindent
(2) Let $a,b\in B$ such that $a\wedge b \in\bigcup_{i\in I}\{0_i\}$. W.l.o.g. let $a\in A_i$ and $b\in A_j$, then $a\wedge^{\B} b = p_{ik}(a)\wedge^{\A_{k}} p_{jk}(b) = 0_{k}$, with $k=i\lor j$. Consequently, $ s(a\lor^{\B} b) = s (p_{ik}(a)\lor^{\A_{k}} p_{jk}(b)) = s_k(p_{ik}(a)\lor^{\A_{k}} p_{jk}(b))  = s_k(p_{ik}(a)) + s_k(p_{jk}(b))  = s_i(a) + s_j(b)  = s(a) + s(b) $, where we have essentially used (finite) additivity of probability measures. 
\end{proof}

\begin{example}\label{ex: stato sull'esempio}
Consider the involutive bisemilattice $\B$ introduced in Example \ref{ex: operazioni Plonka} and define the map $s\colon\B\to [0,1]$ as follows: 

\[ 
s(1) = s(1_{l}) = 1, \text{ for any } l\in\{i,j,k\},
\]
\[
s(0) = s(0_{l}) = 0, \text{ for any } l\in\{i,j,k\},
\]
\[ 
s(a) = s(a') = \frac{1}{2},
\]
\[ 
s(b) = \frac{1}{3}, \; s(b') = \frac{2}{3},
\]
\[
s(c)= s(c')= \frac{1}{2}, \; s(d) = \frac{1}{6}, \; s(e) = \frac{1}{3}, \; s(d') = \frac{5}{6}, \; s(e') = \frac{2}{3}.
\]
It is not difficult to check that $s$ is a (faithful) state over $\B$. Indeed, it is immediate to see that the restrictions of $s$ to the Boolean components of the P\l onka sum are finitely additive probability measures which are, moreover, preserved by homomorphisms of the P\l onka sum.
\qed
\end{example}

Observe that, in Proposition \ref{prop: restrizione su Ai e uno stato}, it is not necessary to assume the existence of a state over the involutive bisemilattice $\B$ (or, over all the Boolean algebras in the P\l onka sum). However, it already gives a gist on which involutive bisemilattices actually carry a state: those whose P\l onka sum representation contains no trivial Boolean algebra, as the trivial Boolean algebra carries no (finitely additive) probability measure (this is shown in Corollary \ref{cor: esistenza dello stato}).

\begin{theorem}\label{th: corrispondenza stati ISBL e booleanizzazione}
Let $\B\in\IBSL$. There exists a bijection $\Phi\colon\mathcal{S}(\B)\to\mathcal{S}(\bool)$ such that $\Phi$ is state preserving, i.e. $s(b) = \Phi(s)(\pi(b))$.
\end{theorem}
\begin{proof}
Consider $\Phi\colon\mathcal{S}(\B)\to\mathcal{S}(\bool)$ defined as follows: 
\begin{equation}\label{ed: definition di Phi}
s(b)\mapsto \Phi(s)= m_{\infty}([b]_{\sim})\coloneqq s(b),
\end{equation}
for every $s\in\mathcal{S}(\B)$ and $b\in B$, a representative of the equivalence class $[b]_{\sim}$. Observe that the definition of $\Phi$ does not depend on the choice of the representative of $[a]_{\sim}$. Indeed, let $a, b\in B$ with $a\neq b$, such that $a,b\in[a]_{\sim}$. W.l.o.g. assume that $a\in A_{i}$ and $b\in A_{j}$, for some $i,j\in I$. Since $a \sim b$, then there exists $k\in I$ with $i,j\leq k$ such that $p_{ik}(a) = p_{jk}(b)$. Then, by Proposition \ref{prop: restrizione su Ai e uno stato}, $s(a) = s_i(a) = s_k(p_{ik}(a)) = s_k(p_{jk}(b)) = s_j(b) = s(b)$.   \\
\noindent
Let us prove that $\Phi(s)=m_{\infty}$ is indeed a (finitely additive) probability measure over the Boolean algebra $\bool$. To this end, observe that $\bigcup_{i\in I}1_{i}\subseteq [1]_{\sim}$ and choose $1_{j}$ (for some $j\in I$) as representative for $[1]_{\sim}$. Then $m_{\infty}([1]_{\sim})= s(1_{j}) = 1$, by Proposition \ref{prop: basic properties}. Let $[a]_{\sim}, [b]_{\sim}\in A_{\infty}$ two elements such that $[a]_{\sim}\wedge [b]_{\sim} = [0]_{\sim}$. W.l.o.g. we can assume that $a\in A_{i}$, $b\in A_{j}$, for some $i,j\in I$, with $i\neq j$, and set $k = i \vee j$. The assumption that $[a]_{\sim}\wedge [b]_{\sim} = [0]_{\sim}$ implies that there exists $l\geq i,j$ such that $p_{il}(a)\wedge p_{jl}(b) = 0_{l}$ (and, obviously, $k\leq l$). Then: 
\begin{align*}
m_{\infty}([a]_{\sim}\vee [b]_{\sim}) & = m_{\infty}([a\vee b]_{\sim}) \\
& = s(a\lor b) \\
& = s_k (p_{ik}(a)\lor^{\A_k} p_{jk}(b))  \\ 
& = s_l (p_{kl} (p_{ik}(a)\lor^{\A_k} p_{jk}(b)) ) \\
& = s_l (p_{il}(a)\lor^{\A_k} p_{jl}(b))  \\
& = s_l(p_{il}(a)) + s_l ( p_{jl}(b)) & \text{(Prop. \ref{prop: restrizione su Ai e uno stato}) } \\
& = s_i (a) + s_j (b)  & \text{(Prop. \ref{prop: restrizione su Ai e uno stato}) }\\
& = s (a) + s (b) \\
& = m_{\infty}([a]_{\sim})+m_{\infty}([b]_{\sim}).
\end{align*}
To complete the proof, let us check that $\Phi$ is invertible. Let $m\colon\bool\to[0,1]$ be a (finitely additive) probability measure over $\bool$. Define $\Phi^{-1}(m[a]_{\sim}) = s_i(a)\coloneqq m([a]_{\sim})$, for any $[a]\in A_{\infty}$, with $a\in A_{i}$ (for $i\in I$). Let us check that $s_i\colon\A_{i}\to [0,1]$ is a (finitely additive) probability measure (over $\A_{i}$). $m_{i}(1_{i}) = m([1_{i}]_{\sim}) = m([1]_{\sim}) = 1$. Moreover, let $a,b\in A_{i}$ two elements such that $a\wedge b = 0_{i}$. Then, $[a]_{\sim}\wedge [b]_{\sim} = [a\wedge b]_{\sim} = [0_{i}]_{\sim} = [0]_{\sim}$. Therefore, $s_i(a\vee b) = m([a\vee b]_{\sim}) = m([a]_{\sim}\vee [b]_{\sim}) = m([a]_{\sim}) + m([b]_{\sim}) = m_{i}(a) + m_{i}(b)$. Observe that, the homomorphism $p_{ij}$ preserves the measures $s_i$ and $s_j$, for any $i\leq j$. Indeed (for $[a]_{\sim}\in A_{\infty}$) $a\sim p_{ij}(a)$, thus $s_i(a) = m([a]_{\sim}) = m([p_{ij}(a)]_{\sim}) = s_j(p_{ij}(a))$. Therefore, since $m_{i}$ is a finitely additive probability measure over $\A_{i}$, for each $i\in I$, and homomorphisms preserve the measures, by Proposition \ref{prop: restrizione su Ai e uno stato}, we get that $\Phi^{-1}(m)$ is a state over $\B$ (it is immediate to check that $\Phi^{-1}$ is the inverse of $\Phi$). Finally, it follows from the definition that $\Phi$ is state preserving.
\end{proof}
\begin{corollary}\label{cor: esistenza dello stato}
Let $\B\in\mathcal{IBSL}$ with $\PL(\A_{i})$ its P\l onka sum representation. The following are equivalent:
\begin{enumerate}
\item $\B$ carries a state; 
\item $\{\A_i\}_{i\in I}$ contains no trivial algebra.
\end{enumerate}
\end{corollary}
\begin{proof}
In virtue of the correspondence established in Theorem \ref{th: corrispondenza stati ISBL e booleanizzazione}, $\B$ carries a state if and only if the Booleanisation $\bool$ does carry a (finitely additive) probability measure, i.e. $\bool$ is non trivial (since the trivial Boolean algebra carries no such map). Moreover, observe that $\bool$ is trivial if and only if there is an algebra $\A_{k}$ (for some $k\in I$) in the system that is trivial. To show the latter claim, suppose that $\bool$ is trivial, i.e. $[0]_{\sim} = [1]_{\sim}$, then there is some $k\in I$ such that $p_{ik}(0_{i}) = p_{jk}(1_{j})$ (for some $i,j\in I$), i.e. $0_{k} = 1_{k}$, showing that $\A_{k}$ is the trivial algebra. For the converse direction, suppose that $\bool$ is non-trivial, i.e. $[0]_{\sim}\neq [1]_{\sim}$, and, by contradiction, that $\{\A_i\}_{i\in I}$ contains a trivial algebra $\A_{k}$ (for some $k\in I$), i.e. $0_{k} = 1_{k}$, but $0_{k}\in [0]_{\sim}$ and $1_{k}\in [1]_{\sim}$, a contradiction with $[0]_{\sim}\neq [1]_{\sim}$.
\end{proof}

Interestingly, the class of involutive bisemilattices admitting no trivial algebra in the P\l onka sum representation forms a subquasivariety of $\IBSL$, known as $\mathcal{NGIB}$,  which has been introduced in \cite{Paoliextensions} and plays the role of the algebraic counterpart of the non-paraconsistent extension of the logic PWK by adding the ex-falso quodlibet (this consists of the unique extension of PWK, different from classical logic). $\mathcal{NGIB}$ is axiomatised\footnote{It shall be mentioned that, in \cite{Paoliextensions}, the authors work with generalised involutive bisemilattices, namely defined in the type containing no constant. However, this difference is not significant for the purpose of our discussion.} (with respect to $\IBSL$) by the quasi-identity $x\approx x'\Rightarrow y\approx z$.

From now on, we will consider only (non-trivial) involutive bisemilattices whose P\l onka sum representation contains no trivial Boolean algebra, thus the non-trivial elements belonging to the class $\mathcal{NGIB}$. Hence, in view of Corollary \ref{cor: esistenza dello stato}, any involutive bisemilattice considered carries at least a state.
We indicate by $\mathcal{S}(\B)$ and $\mathcal{S}(\bool)$ the spaces of states and of (finitely additive) probability measures of an involutive bisemilattice $\B$ and of its Booleanisation $\bool$, respectively. 

The correspondence established in Theorem \ref{th: corrispondenza stati ISBL e booleanizzazione} allows to provide an integral representation of states over an involutive bisemilattice. 
Let $\bool$ be the Booleanisation of an involutive bisemilattice $\B$ and $\bool^{\ast}$ the Stone space dually equivalent to $\bool$. It is useful to recall that $\bool^{\ast}$ corresponds to the \emph{inverse limit} over the direct system whose elements $
\{\A^{\ast}_{i}\}_{i\in I}$ are the dual spaces of the Boolean algebras $\{\A\}_{i\in I}$ in the representation of $\B$ (see \cite{haimo}).

\begin{theorem}[Integral representation of states]\label{th: rappresentazione integrale stati IBSL}
Let $\B\in\N$ and $s\colon\B\to [0,1]$ be a state. Then
\begin{enumerate}
\item There is a bijection $\chi\colon \mathcal{S}(\B)\to\mathcal{M}(\bool^{\ast})$, where $\mathcal{M}(\bool^{\ast})$ is the space of all regular Borel probability measures on $\bool^{\ast}$.
\item For every $b\in B $, 
\[ 
s(b) = \int_{\bool^{\ast}} \widehat{[b]}_{\sim}(M) \; d\mu_{s}(M), 
\]
where $\widehat{[b]}_{\sim}$ is the unique function associated to $[b]_{\sim}\in A_{\infty}$, $M\in \bool^{\ast}$ and $d\mu_{s}= \chi(s) $.
\end{enumerate}
\end{theorem}
\begin{proof}
It follows from the bijective correspondence between states of an involutive bisemilattices and of its Booleanisation (Theorem \ref{th: corrispondenza stati ISBL e booleanizzazione}) and integral representation for probability measures over Boolean algebras (Theorem \ref{th: rappresentazione integrale misure di probabilita}). 
\end{proof}

One may wonder whether it is possible to provide a different integral representation of states which makes use of the dual space of an involutive bisemilattice (for instance, the P\l onka product described in \cite{plonkaproduct}) instead of the inverse limits of the dual spaces of the Boolean algebras involved in the P\l onka sum representation. This is a question that we do not address in the present paper.

\section{Faithful states}\label{Sec: faithful states}

Recall that a state $s$ over an involutive bisemilattice $\B$ is faithful (cfr. Definition \ref{def: stato su IBSL}) when $s(a) > 0$, for any $a\neq 0$. By Proposition \ref{prop: restrizione su Ai e uno stato}, this is equivalently expressed by saying that $s(a) > 0$, for every $a\not\in\{ 0_{i}\}_{i\in I}$. A (finitely additive) probability measure $m$ over a Boolean algebra $\C$ is \emph{regular} provided that $m(a) > 0$, when $a\neq 0$. 

The presence of a faithful state $s$ over an involutive bisemilattice $\B$ has a non-trivial consequence on the structure of its P\l onka sum representation, as expressed in the following. 

\begin{proposition}\label{prop: s faithful su Ai}
Let $s$ be a state over $\B\in\IBSL$. The following are equivalent:  
\begin{enumerate}
\item $s$ is faithful; 
\item $s_i\colon \A_{i}\to [0,1]$ is a regular (finitely additive) probability measure over $\A_i$, for every $i\in I$, and $p_{ij}$ is an injective homomorphism preserving the measures, for each $i\leq j$.
\end{enumerate}
\end{proposition} 
\begin{proof}
We just show the non-trivial direction (1) $\Rightarrow $ (2). By Proposition \ref{prop: restrizione su Ai e uno stato}, we only have to prove that $s_i $ is regular and that $p_{ij}$ is an injective homomorphism, for every $i\leq j$. The former is immediate, indeed for $a\in A_{i}$, (with $i\in I$) such that $a\neq 0_{i}$, then $s_i (a) = s (a) > 0 $. As for the latter, let $i\leq j$ for some $i,j\in I$. Let $a\in ker(p_{ij})$, then $p_{ij}(a) = 0_j$ (we are using here the fact that congruences are determined by their $0$-class, see for instance \cite{GivantHalmos}). Then, $s_j(p_{ij}(a)) = s(p_{ij}(a)) = s(0_{j}) = 0$, by Proposition \ref{prop: basic properties}. By Proposition \ref{prop: restrizione su Ai e uno stato}, $p_{ij}$ preserves the measures (and $s_i$, $s_j$ are probability measures), hence $s_i(a)= 0$ and, since $s_i$ is regular (as shown above), then $a=0_{i}$. This shows that $ker(p_{ij}) = \{ 0_{i}\}$, i.e. $p_{ij}$ is injective. 
\end{proof}

\begin{definition}\label{def: IBSL iniettivo}
Let $\B\in\IBSL$. We say that $\B$ is \emph{injective} if,  for every $i,j\in I$ such that $i\leq j$, the homomorphism $p_{ij}\colon\A_{i}\to\A_{j}$ is injective.
\end{definition}

We refer to the class of injective involutive bisemilattices, that has also been introduced in \cite{Paoliextensions}, as $\inj$. 
It is not difficult to see that $\inj$ is closed under subalgebras and products but not under homomorphic images. 
 
Recall that the variety of involutive bisemilattices admits a partition function $\cdot$, that can be defined as $x\cdot y\coloneqq x\wedge(x\vee y)$ (see Examples \ref{ex: operazioni Plonka} and \ref{ex: nuovo esempio}).

\begin{proposition}
Let $\B\in \IBSL$ with partition function $\cdot$. The following are equivalent:
\begin{enumerate}
\item $\B\in\inj$; 
\item $\B\models x\cdot y\approx x \;\&\; y\cdot x\approx y\;\&\; x\cdot z\approx y\cdot z \Rightarrow x\approx y $. 
\end{enumerate}

\begin{proof}
(1) $\Rightarrow$ (2) Suppose $\B\in\inj$, i.e. $p_{ij}$ is an embedding for each $i\leq j$. Suppose, in view of a contradiction, that $\B $ does not satisfy condition (2). Then, there exist elements $a,b,c\in B$ such that $a\cdot b = a$ and $b\cdot a = b$ and $a\cdot c = b\cdot c$ but $a\neq b$. By Theorem \ref{th: Teorema di Plonka}-(1), $a\cdot b = a$ and $b\cdot a = b$ imply that $a,b\in A_{i}$, for some $i\in I$. W.l.o.g. assume that $c\in A_{j}$, for some $j\in I$ and set $k=i\vee j$. Then, applying Theorem \ref{th: Teorema di Plonka} and the assumption that $a\cdot c = b\cdot c$, we have $p_{ik}(a)= p_{ik}(a)\cdot p_{jk}(c) = a\cdot c = b\cdot c = p_{ik}(b)\cdot p_{jk}(c) = p_{ik}(b)$. Since $\B\in\inj$, $p_{ik}$ is injective, namely $a = b$, in contradiction with our hypothesis. \\
\noindent
 (2) $\Rightarrow$ (1) Suppose $\B$ satisfies condition (2) and, by contradiction, that $\B\not\in\inj$, namely there exists a homomorphism $p_{ij}$ (for some $i\leq j$) which is not an embedding. Hence, there exists element $a\neq b$ such that $p_{ij}(a)= p_{ij}(b)$. Clearly $a,b\in A_i$ (otherwise $p_{ij}$ is not well defined), so, by Theorem \ref{th: Teorema di Plonka}-(1), $a\cdot b = a$ and $b\cdot a = b$. Let $c =p_{ij}(a)= p_{ij}(b)$, then $a\cdot c = p_{ij}(a)\cdot p_{ij}(a) = p_{ij}(a) = p_{ij}(b) = p_{ij}(b)\cdot p_{ij}(b) = b\cdot c$. Then, by condition (2), we have that $a = b$, a contradiction. 
\end{proof}
\end{proposition}

It follows from the above Proposition that $\inj$ is a quasi-variety, which can be axiomatised relatively to $\IBSL$ by the quasi-identity $ x\cdot y\approx x \;\&\; y\cdot x\approx y\;\&\; x\cdot z\approx y\cdot z \Rightarrow x\approx y $ (a different quasi-equational axiomatisation can be found in \cite{Paoliextensions}).


\begin{theorem}\label{th: corrispondenza stati ISBL e booleanizzazione (caso faithful)}
Let $\B\in\inj$. Then there is a bijective correspondence between faithful states over $\B$ and regular measures over $\bool$.
\end{theorem}

\begin{proof}
The correspondence is given by the map $\Phi$, defined in \eqref{ed: definition di Phi}. Indeed, let $\B\in\inj$ and $s\colon\B\to [0,1]$ a faithful state. Assume that $[a]_{\sim}\in A_{\infty}$ with $[a]_{\sim}\neq [0]_{\sim}$. Then $a\neq 0$, hence $s(a) > 0$ and $m_{\infty}([a]_{\sim}) = s(a) > 0$, which shows that $\Phi(s)$ is regular. For the other direction, it is immediate to check that, given a regular measure $m\colon\bool\to [0,1]$ then $s_i(a) = \Phi^{-1}(m([a]_{\sim}))$ (for $a\in A_{i}$) is also a regular measure and since $\B$ is injective, Proposition \ref{prop: s faithful su Ai} guarantees that $\Phi^{-1}(m)$ is a faithful state over $\B$.
\end{proof}

Combining \cite[Proposition 3.1.7]{FlaminioKroupa} and Theorem \ref{th: corrispondenza stati ISBL e booleanizzazione} we directly get the following.

\begin{corollary}
The space $\mathcal{S}(\B)$ of states of an involutive bisemilattice $\B$ can be identified (via $\Phi$) with a non-empty compact subspace of $[0,1]^{\bool}$.
\end{corollary}

It is natural to wonder under which conditions an involutive bisemilattice $\B$ (whose P\l onka sum representation contains no trivial algebra) carries a faithful state. Theorem \ref{th: corrispondenza stati ISBL e booleanizzazione (caso faithful)} suggests that this might be the case provided that $\B$ is injective and its Booleanisation $\bool$ actually carries a regular probability measure. In general, as observed in \cite{Kelley59}, not every (non-trivial) Boolean algebra carries a regular probability measure (necessary and sufficient conditions for a Boolean algebra to carry a regular measure are stated in \cite[Theorem 4]{Kelley59}). 

It is possible to define the categories of injective involutive bisemilattices and Boolean algebras ``with faithful state'', ``with regular probability measure'', respectively. Objects are pairs $(\B,s)$ and ($\A, m$), where $\B\in\IBSL$, $\A $ is a Boolean algebra, $s$ is a state over $\B$ and $m$ is a (finitely additive) probability measure over $\A$. A morphism between two objects $(\B_{1}, s_{1})$ and $(\B_{2}, s_{2})$ is a homomorphism (between the corresponding algebras in the first component) which preserves the measures, namely $s_{1}(b) = s_{2}(h(b))$, for every $b\in B_{1}$ and every homomorphism $h\colon\B_{1}\to\B_{2}$. It is immediate to check that involutive bisemilattices carrying a faithful state (Boolean algebras carrying a probability measure, resp.) form a category, which we indicate by the pair $\langle\IBSL, \mathcal{S}(\B)\rangle$ ($\langle\BA, \mathcal{S}(\A)\rangle$, resp.).

Let us define a functor $\mathcal{F}\colon \langle\IBSL, \mathcal{S}(\B)\rangle\to \langle\BA, \mathcal{S}(\A)\rangle$, which associates, to each object $(\B,s)$, the object $\mathcal{F}(\B,s) = (\bool, \Phi(s))$, with $\Phi$ defined as in \eqref{ed: definition di Phi} and, to each morphism $h\colon(\B_{1},s_{1})\to(\B_{2},s_{2})$, the morphism $\mathcal{F}(h) = \overline{h} $ with $ \overline{h}\colon\A_{1_{\infty}}\to\A_{2_{\infty}}$, defined as $\overline{h}[a_{1}] = [h(a_{1})]$.  

\begin{theorem}
$\mathcal{F}$ is a covariant functor between the categories of injective involutive bisemilattices with faithful states and Boolean algebras with regular probability measures. 
\end{theorem}
\begin{proof}
Let $(\B,s)$ an object in $\langle\IBSL, \mathcal{S}(\B)\rangle$. Then, by Theorem \ref{th: corrispondenza stati ISBL e booleanizzazione}, $\mathcal{F}(\B,s) = (\bool, \Phi(s))$ is an object in $\langle\BA, \mathcal{S}(\A)\rangle$. We have to prove that, for every homomorphism $h\colon(\B_{1},s_{1})\to(\B_{2},s_{2})$ preserving states, the map $ \overline{h}\colon(\A_{1_{\infty}}, \Phi(s_{1}))\to(\A_{2_{\infty}}, \Phi(s_{2}))$ is a Boolean homomorphism, which preserves the (finitely additive) probability measures.  
To see that $\overline{h}$ is indeed a Boolean homomorphism, we show just one case (all the others are proved analogously). Let $[a],[b]\in\A_{1_{\infty}}$, then $\overline{h}([a]\wedge[b]) = \overline{h}([a\wedge b]) = [h(a\wedge b)] = [h(a)\wedge h(b)] = [h(a)]\wedge [h(b)] = \overline{h}(a)\wedge\overline{h}(b)$. 
Moreover, $\overline{h}$ preserves the states. Indeed let $[a_{1}]\in A_{1_{\infty}}$ and $\Phi(s_{1})\in \mathcal{S}(\A_{1_{\infty}})$. Then $\Phi(s_{1})([a_{1}]) = s_{1}(a_{1}) = s_{2}(h(a_{1})) = \Phi(s_{2})([h(a_{1})]) = \Phi(s_{2})(\overline{h}[a_{1}])$, where we have used the fact that $h$ preserves states.
Finally, it follows from the definition of $\overline{h}$ that the following diagram is commutative and this concludes the proof of our claim.
\begin{center}
\begin{figure}[h]
\begin{tikzpicture}[scale=0.8]
\draw (-4,0) node {$(\A_{1_{\infty}}, \Phi(s_{1}))$};
	\draw (4,0) node {$(\A_{2_{\infty}}, \Phi(s_{2}))$};
	
	\draw [line width=0.8pt, ->] (-2,0) -- (2,0);
	\draw (0,-0.4) node {\begin{footnotesize}$\overline{h}$\end{footnotesize}};
	
	\draw [line width=0.8pt, <-] (2.8,3) -- (-2.8,3);
	\draw (0, 3.3) node {\begin{footnotesize}$h$\end{footnotesize}};

	\draw (-4,3) node {$(\B_{1},s_{1})$}; 
	
	\draw (4,3) node {$(\B_{2},s_{2})$};
		
	\draw [line width=0.8pt, ->] (-4,2.6) -- (-4,0.4);
	\draw (-4.8,1.7) node {\begin{footnotesize}$(\pi,\Phi)$\end{footnotesize}};
	\draw (4.8,1.7) node {\begin{footnotesize}$(\pi,\Phi)$\end{footnotesize}};
	\draw [line width=0.8pt, <-] (4,0.4) -- (4,2.6);
	
\end{tikzpicture}
\end{figure}
\end{center}
\end{proof}

The functor $\mathcal{F}$ admits many adjoints, depending on the number of injective involutive bisemilattices having the same Booleanisation. \\

\noindent
\textbf{Problem:} Characterise all the injective involutive bisemilattices having the same Booleanisation.

The problem is not of easy solution. Taking up the suggestion of an anonymous reviewer, we have drawn some considerations about it in Appendix B.


\section{States, metrics and completion}\label{Sec: metric and completition}

\begin{definition}\label{def: pseudo-metric}
Let $X$ be a set. A \emph{pseudometric} on $X$ is a map $d\colon X\times X\to \mathbb{R}$ such that: 
\begin{enumerate}
\item $d(x,y)\geq 0$,
\item $d(x,x) = 0 $,
\item $d(x,y) = d(y, x)$, 
\item $ d(x,z)\leq d(x,y) + d(y,z) $ (triangle inequality),
\end{enumerate}
for all $x,y,z\in X $.
\end{definition}

A pair $(X,d)$ given by a set with a pseudometric $d$ is called \emph{pseudometric space}. A pseudometric $d$, on $X$, is a \emph{metric}, if $d(x,y) = 0$ implies that $x = y$. Notice that, as for metric spaces, the triangle inequality is enough to show that the pseudo-metric $d$ is continuous (a fact that we will use in several proofs), when $X$ is topologised with the topology induced by the pseudo-metric (with no explicit mention).

Let us recall that in any Boolean algebra $\A$ it is possible to define the \emph{symmetric difference} $\vartriangle\colon A\times A\to A$, $a\vartriangle b\coloneqq (a\wedge b')\vee(a'\wedge b)$. The presence of a (finitely additive) probability measure $m$ over a Boolean algebra $\A$ allows to define a pseudo-metric $d\coloneqq m\circ\vartriangle$ on $\A$, which becomes a metric, in case $m$ is regular (see \cite{Kolmogorov} for details).


The symmetric difference can be (analogously) defined also for an involutive bisemilattice $\B$. Let $a,b\in B$, with $\B\cong\PL(\A_i )$, $a\in A_{i}$ and $b\in A_{j}$, for some $i,j\in I$. Then $a\vartriangle b = (a\wedge^{\B}b')\vee^{\B}(a'\wedge^{\B} b) = p_{ik}(a)\vartriangle^{\A_{k}}p_{jk}(b)$, where $k=i\vee j$. Clearly, in case $\B$ carries a state $s$, then one can define the map $d_{s}\colon\B\to [0,1] $ as:
\begin{equation}\label{eq: definizione ds}
d_{s} \coloneqq s\circ \vartriangle.
\end{equation}

We will sometimes refer to the (pseudo)metrics induced by states and finitely additive probability measures as distances.

\begin{proposition}\label{prop: pseudo-metrica su IBSL}
Let $\B$ be an involutive bisemilattice carrying a state $s$. Then $d_{s} $ is a pseudo-metric on $B$. 
\end{proposition}
\begin{proof}
(1) obviously holds, since $s\colon\B\to [0,1]$. \\
\noindent
(2) Let $a\in A_{i}$, for some $i\in I$. Then $ d_{s}(a,a) = s((a\land a')\vee (a'\land a))= s(0_{i}\vee 0_{i}) = s (0_{i})= 0$, by Proposition \ref{prop: basic properties}.\\
\noindent
(3) holds since $\vee$ and $\land$ are commutative operations. \\
\noindent
(4) Let $a\in  A_{i}$, $b\in A_{j}$ and $c\in A_{k}$ with $i\neq j \neq k$. Preliminarily, observe that
\begin{align*}
d_{s} (a,c) & = s((a\land^{\B}c')\vee^{\B} (a'\land^{\B} c)) \\
& = s((p_{i,i\vee k}(a)\land^{\A_{i\vee k}} p_{k,i\vee k}(c)')\vee^{\A_{i\vee k}} (p_{i,i\vee k}(a)'\land^{\A_{i\vee k}} p_{k,i\vee k}(c))) \\
& = s_{i\vee k}((p_{i,i\vee k}(a)\land^{\A_{i\vee k}} p_{k,i\vee k}(c)')\vee^{\A_{i\vee k}} (p_{i,i\vee k}(a)'\land^{\A_{i\vee k}} p_{k,i\vee k}(c)))  \\ 
& = s_{i\vee k\vee j} (p_{i\vee k ,i\vee k \vee j}((p_{i,i\vee k}(a)\land^{\A_{i\vee k}} p_{k,i\vee k}(c)')\vee^{\A_{i\vee k}} (p_{i,i\vee k}(a)'\land^{\A_{i\vee k}} p_{k,i\vee k}(c)))) & \text{(Prop. \ref{prop: restrizione su Ai e uno stato}) } \\
& = s_{i\vee k\vee j} (a\vartriangle^{\A_{i\vee k\vee j}}c). 
\end{align*}

Similarly, it is possible to show that $ d_{s}(a,b) = s_{i\vee k\vee j} (a\vartriangle^{\A_{i\vee k\vee j}}b) $ and $ d_{s}(b,c) = s_{i\vee k\vee j}(b\vartriangle^{\A_{i\vee k\vee j}} c) $. Hence $d_{s}(a,c)= s_{i\vee k\vee j} (a\vartriangle^{\A_{i\vee k\vee j}} c) \leq s_{i\vee k\vee j} (a \vartriangle^{\A_{i\vee k\vee j}} b) + s_{i\vee k\vee j} (b \vartriangle^{\A_{i\vee k\vee j}}c) = s\circ d (a,b) + s\circ d (b,c) = d_{s}(a,b) + d_{s}(b,c)$.
\end{proof}

\begin{remark}
Recall that, for Boolean algebras (and, more in general, MV-algebras \cite{metricMV}), the choice of a regular probability measure $m$ (or, a faithful state in the case of MV-algebras) implies that the induced pseudometric $d_{m}$ is indeed a metric. This fact does not hold for involutive bisemilattices, where the choice of a faithful state $s$ is not enough to guarantee that $d_{s}$ is a metric. Indeed, consider any $\B\in\N$ and two distinct elements $a,b\in B$, $a\in A_{i}$, $b\in B_{j}$, such that $p_{i,i\vee j}(a) = p_{j,i\vee j}(b)$. Then $d_{s}(a,b) = s (a\vartriangle^{\B}b) = s(p_{i,i\vee j}(a) \vartriangle^{\A_{i\vee j}} p_{j,i\vee j}(b)) = s(0_{i\vee j}) = 0$; however, $a\neq b$. Notice, moreover, that $a\sim b$, i.e. they belong to the same equivalence class in the Booleanisation $\bool$. Indeed, it is easy to check that $d_{s}$ is a metric over $\B$ if and only if $\B=\bool$. 
\end{remark}

We now recall a well known fact about Boolean algebras.

\begin{lemma}\label{lem: fatti noti algebre di Boole con misura}
Let $\A$ be a Boolean algebra carrying a finitely additive probability measure $m$ and let $d_{m}$ be the pseudometric ($d_{m}=m\circ\Delta$) defined via $m$. Then:
\begin{enumerate}
\item $d_{m}(a,b) = d_{m}(a',b')$;
\item $d_{m}(a\vee b,c\vee d) \leq d_{m}(a,c) + d_{m}(b,d)$, 
\end{enumerate}
for any $a,b,c,d\in A$.
\end{lemma}
\begin{proof}
The proof can be adapted from \cite[Lemma 3.1]{metricMV}, observing that Boolean algebras are MV-algebras where the operations $\oplus$ and $\vee$ coincide.
\end{proof}

We decide to give explicit proofs of all the following results in the particular case of pseudometric (injective) involutive bisemilattices (and relative Booleanisations), although some of them could be derived from the general theory of pseudometric spaces (see for instance \cite{Kelleybook}).

In analogy to what is done in \cite{metricMV} for MV-algebras, in the remaining part of this section we study the metric completion for involutive bisemilattices. The completion is a standard construction for metric spaces (see \cite{Kelleybook}), which can be analogously applied to pseudo-metric spaces. 

Recall that a (psudo)metric space is \emph{complete} if every Cauchy sequence is convergent. Given a pseudometric space $(X,d)$, a completion $(\widehat{X},\widehat{d})$ of $(X,d)$ is such that: 
\begin{enumerate}
\item $(\widehat{X},\widehat{d})$ is complete; 
\item there exists an isometric injective map $j\colon X\to\widehat{X} $ such that $j(X)$ is dense in $\widehat{X}$. 
\end{enumerate}

Observe that the second condition means that, for every $\widehat{x}\in\widehat{X}$, there exists a sequence $x_{n}\in X$ such that $j(x_n)\to \widehat{x}$.

In the sequel, we can assume, up to isometry, that the embedding $j$ is the natural inclusion $X \hookrightarrow\widehat{X}$. Notice that the completion $(\widehat{X}, \widehat{d})$ is uniquely determined, by the above conditions (1)-(2), up to isometries.

We proceed in the same way for involutive bisemilattices. In detail, given $(\B,d_{s})$ a pair where $\B$ is an involutive bisemilattice (carrying a state $s$) and $d_{s}$ the pseudometric induced by the state $s$, we can associate to it, on the one hand, the completion $(\widehat{\B},\widehat{d_{s}})$; on the other hand, we can consider its P\l onka sum representation $\PL(\A_{i})$. By Proposition \ref{prop: restrizione su Ai e uno stato}, for every $i\in I$, $(\A_{i},d_{s_i})$ is a pseudo-metric space (since $s_{i}$ is a probability measure over $\A_{i}$), which is metric in case $s$ is faithful (see Proposition \ref{prop: s faithful su Ai}). Therefore, it makes sense to consider the pseudo-metric space $(\widetilde{\B},\widetilde{d})$, where $\widetilde{\B}=\PL(\widehat{\A}_{i})_{i\in I}$ (the P\l onka sum of the completions of the Boolean algebras\footnote{The fact that the completion of a Boolean algebra is still a Boolean algebra is a routine exercise: one may follow the strategy applied to MV-algebras in \cite{metricMV} (using the fact that  Boolean algebras are MV-algebras where $\oplus$ is idempotent).} in the system representing $\B$), $\widetilde{d}(\widetilde{a},\widetilde{b}) = \lim_{n\to\infty} d_{s}(a_{n},b_{n})$, where $\widetilde{a}\in\widehat{A}_{i}$, $\widetilde{b}\in\widehat{A}_{j}$ (for some $i,j\in I$), and $a_{n},b_{n}$ are sequences (of elements) in $A_{i}, A_{j}$, respectively, such that $a_{n}\to\widetilde{a}$, $b_{n}\to\widetilde{b}$.
We are going to show (see Theorem \ref{th: sulle completitions} below) that $\widetilde{\B}\in\IBSL $ and, moreover, that $(\widetilde{\B},\widetilde{d})$ is the completion of $(\B,d_{s})$.


\begin{lemma}\label{lemma: omomorfismi tra completizzazione di algebre di Boole}
Let $(\A_{1},d_{1}) $, $ (\A_{2},d_{2})$ be two Boolean algebras with distance (induced by a probability measure) and $h\colon\A_{1}\to \A_{2}$ be a distance preserving homomorphism. Then there exists a distance preserving homomorphism  
$\widehat{h}\colon\widehat{\A}_{1}\to\widehat{\A}_{2}$ such that $\widehat{h}_{|\A_{1}} = h$.
\end{lemma}
\begin{proof}
Let $\widehat{a}\in\widehat{A}_{1}$ and define $\widehat{h}\colon\widehat{\A}_{1}\to\widehat{\A}_{2}$ as 
\[
\widehat{h}(\widehat{a})\coloneqq \lim_{n\to \infty}h(a_{n}), 
\]
where $a_{n}\to \widehat{a}$ is a sequence of (elements of) $A_{1}$ convergent to $\widehat{a}$. Observe that $\widehat{h}$ is well defined. Indeed, suppose that $a'_{n}\to \widehat{a}$, i.e. $a'_{n}$ is a different sequence convergent to the element $\widehat{a}$. Then $\lim_{n\to\infty}\di_{2}(h(a_{n}), h(a'_{n})) = \lim_{n\to\infty}d_{2}(h(a_{n}), h(a'_{n}))=\lim_{n\to\infty}d_{1}(a_{n}, a'_{n}) = \di_{1}(\widehat{a}, \widehat{a}) = 0$, where the second equality is justified by the fact that $h$ is an isometry (distance preserving map). This shows that $\lim_{n\to\infty}h(a_{n}) = \lim_{n\to\infty}h(a'_{n})$. Moreover, it follows by construction that $\widehat{h}_{|\A_{1}} = h$. It only remains to show that $\widehat{h}$ is distance preserving and a Boolean homomorphism. Let $\widehat{a},\widehat{b}\in\widehat{A}_{1}$. Then 
\begin{align*}
\di_{2}(\widehat{h}(\widehat{a}),\widehat{h}(\widehat{b}))& =  \di_{2}(\lim_{n\to\infty} h(a_{n}), h(b_{n}))\\ 
& = \lim_{n\to\infty} \di_{2}(h(a_{n}), h(b_{n})) & (\di_{2}\text{ is continuous}) \\
& = \lim_{n\to\infty} d_{2}(h(a_{n}), h(b_{n})) & \\
& = \lim_{n\to\infty} d_{1}(a_{n}, b_{n}) & (h \text{ isometry}) \\
& = \di_{1}(\widehat{a}, \widehat{b}).
\end{align*}
To see that $\widehat{h}$ is a Boolean homomorphism, we only show the cases of negation and one of the two binary operations. 
\begin{align*}
\widehat{h}(\widehat{a'})& =  \widehat{h}(\lim_{n\to\infty} a_{n}')\\ 
& = \lim_{n\to\infty} \widehat{h} (a_{n}') & (\widehat{h} \text{ is continuous}) \\
& = \lim_{n\to\infty} h (a_{n}') & (\widehat{h}_{|\A_{1}} = h) \\
& = \lim_{n\to\infty} h (a_{n})' & (h \text{ homomorphism}) \\
& = \widehat{h}(\widehat{a})',
\end{align*}
where the last equality holds, since, by Lemma \ref{lem: fatti noti algebre di Boole con misura}-(1), $\di_{2}(h(a_{n})',\widehat{h}(\widehat{a})') = \di_{2}(h(a_{n}),\widehat{h}(\widehat{a})) = 0  $. We reason similarly for $\vee$:

\begin{align*}
\widehat{h}(\widehat{a}\vee \widehat{b})& =  \widehat{h}(\lim_{n\to\infty} a_{n}\vee b_{n})\\ 
& = \lim_{n\to\infty} \widehat{h} (a_{n}\vee b_{n}) & (\widehat{h} \text{ is continuous}) \\
& = \lim_{n\to\infty} h (a_{n}\vee b_{n}) & (\widehat{h}_{|\A_{1}} = h) \\
& = \lim_{n\to\infty} h (a_{n})\vee h( b_{n}) & (h \text{ homomorphism}) \\
& = \widehat{h}(\widehat{a})\vee\widehat{h}( \widehat{b}),
\end{align*}
where the last equality holds, since, by Lemma \ref{lem: fatti noti algebre di Boole con misura}-(2), $\di_{2}(h(a_{n})\vee h(b_{n}),\widehat{h}(\widehat{a})\vee \widehat{h}(\widehat{b})) \leq \di_{2}(h(a_{n}), \widehat{h}(\widehat{a})) +\di_{2}(h(b_{n}),\widehat{h}(\widehat{b}) ) = 0.  $
\end{proof}

\noindent
\textbf{Notation:} given the pseudo-metric space $(\B,d_{s})$, where $\B\in\N$ and $d_{s}$ is the pseudo-metric induced by a state $s$, we indicate by $d_{\infty}$ (instead of $d_{\Phi(s)}$) the pseudo-metric on its Booleanisation, obtained via the bijection in Theorem \ref{th: corrispondenza stati ISBL e booleanizzazione}. 

\begin{lemma}\label{lemma: B completo sse bool e completa}
Let $\B\in\IBSL$ carrying a state $s$. Then $(\B,d_{s})$ is complete if and only if $(\bool,d_{\infty})$ is complete.
\end{lemma}
\begin{proof}
Observe that a sequence $\{x_{n}\}_{n\in\mathbb{N}}$ of elements in $\B$ is a Cauchy sequence if and only if the sequence $\{[x_{n}]_{\sim}\}_{n\in\mathbb{N}}$ is a Cauchy sequence in $\bool$. Indeed, if, for each $\varepsilon > 0$, there is a $n_{0}\in\mathbb{N}$ such that $d_{s}(x_{n},x_{m}) < \varepsilon $, for every $n,m > n_{0}$, then also $d_{\infty}([x_{n}]_{\sim},[x_{m}]_{\sim}) < \varepsilon $, since $d_{s}(x_{n},x_{m}) = d_{\infty}([x_{n}]_{\sim},[x_{m}]_{\sim})$ 
by Theorem \ref{th: corrispondenza stati ISBL e booleanizzazione}. The result follows by observing that $\lim_{n\to\infty} x_{n} = x$ if and only if $\lim_{n\to\infty} [x_{n}]_{\sim} = [x]_{\sim}$.
\end{proof}

\begin{theorem}\label{th: sulle completitions}
Let $(\B, d_{s})$ be an involutive bisemilattice with the pseudo-metric $d_{s}$ induced by a state $s$, then: 
\begin{enumerate}
\item $\widetilde{\B}=\PL(\widehat{\A}_{i})$ is an involutive bisemilattice; 
\item $(\widehat{\B},\widehat{d}_{s}) $ is isometric to $(\widetilde{\B}, \widetilde{d})$.
\end{enumerate}
\end{theorem}
\begin{proof}
(1) $\widetilde{B}=\bigsqcup_{i\in I}\widehat{A}_{i}$, where $\widehat{\A}_{i}$, for each $i\in I$, is a Boolean algebra (since Boolean algebras are closed under completions). Moreover, by Lemma \ref{lemma: omomorfismi tra completizzazione di algebre di Boole}, the system $\langle \{\widehat{\A}_{i}\}_{i\in I}, I, \{\widehat{h}_{ij}\}_{i\leq j}\rangle$ is a semilattice direct systems of Boolean algebras (given that $\langle \{ \A_{i}\}_{i\in I}, I, \{ h_{ij}\}_{i\leq j}\rangle$ is such) and this is enough to conclude that $\widetilde{\B}$ is an involutive bisemilattice.\\
\noindent
(2) Preliminarily observe that $\widetilde{\A}_{\infty} = \widehat{\A}_{\infty}$ (the Booleanisations of $\widetilde{\B}$ and $\widehat{\B}$ coincide). Then, by Lemma \ref{lemma: B completo sse bool e completa}, it follows that $\widetilde{\B}$ is complete, as its Booleanisation is complete. We claim that $(\B,d_{s})$ is a dense subset of $(\widetilde{\B},\widetilde{d})$. Indeed, let $\widetilde{b}\in\widetilde{B}$, then $\widetilde{b}\in\widehat{A}_{i}$, for some $i\in I$ (since $\widetilde{B}= \bigsqcup_{i\in I}\widehat{A}_{i} $). Thus, there exists a sequence $\{a_{n}\}_{n\in\mathbb{N}}\in A_{i}$ such that $\lim_{n\to\infty}d_{s}(a_{n},\widetilde{b}) = 0$. Then, by definition of $\widetilde{d}$, $\widetilde{d}(a_{n},\widetilde{b})\to 0$, which implies that $a_{n}$ (as element of $B$) converges to $\widetilde{b}$, i.e. $\B$ is dense in $\widetilde{\B}$. 
It follows (from the previous claim) that there is an isometric bijection $f\colon (\widehat{\B},\widehat{d}_{s})\to (\widetilde{\B}, \widetilde{d})$. 
\end{proof}

\section{The topology of involutive bisemilattices}\label{Sec: topology}

An involutive bisemilattice $\B$ carrying a state $s$ can be topologised with the topology induced by the pseudo-metric $d_{s}$, defined in \eqref{eq: definizione ds}. In virtue of Theorem \ref{th: corrispondenza stati ISBL e booleanizzazione} (and Theorem \ref{th: corrispondenza stati ISBL e booleanizzazione (caso faithful)} for faithful states), also the Booleanisation $\bool $ (of $\B$) can be topologised via the corresponding probability measure which we indicate as $d_{\infty}= \Phi(s)\circ\vartriangle$ (where $\Phi$ is the map defined in \eqref{ed: definition di Phi}). In this section, we confine ourselves only to faithful states over (injective) involutive bisemilattices and when referring to $\B$ and $\bool$ as topological spaces, we think them as equipped with the topologies $\mathcal{T}_{d_{s}}$ and $\mathcal{T}_{d_{\infty}}$ induced by the respective (pseudo) metric. To achieve this we should assume that the involutive bisemilattices under consideration carry faithful states.
Recall that, in this topology, a subset $U\subset B$ is open if and only if for every $x\in U$, there exists $r > 0$ such that $D_{r}(x)\subset U$, where $D_{r}(x)= \{y\in B | d_{s}(x,y) < r \}$ is the open disk centered in $x$ with radius $r$. Moreover, one base of both topologies is given by the family of all open disks with respect to $d_{s}$ and $d_{\infty}$, respectively.

\begin{remark}\label{rem: lo stato e sempre una funzione continua}
Notice that, in an involutive bisemilattice $\B$, carrying a state $s$, $s(b)=d_{s}(b,0)$, for every $b\in B$. Indeed let $b\in A_{i}$, for some $A_{i}$ in the P\l onka sum representation of $\B$, then $d_{s}(b,0)= s\circ\Delta^{\B}(b,0) = s ((b\wedge 1_{i})\vee(b'\wedge 0_{i})) = s(b\vee 0_{i}) = s(b)$. This implies that $s\colon\B\to [0,1]$ is continuous (since the pseudo-metric is continuous) when $\B$ is topologised via  $\mathcal{T}_{d_{s}}$.
\end{remark}

\begin{definition}
Let $X$ be a topological space and let $\equiv\;\subseteq X\times X$ the equivalence relation defined as $x\equiv y$ if and only if $x$ and $y$ have the same open neighbourhoods. Then, the space $X_{/\equiv}$ is the \emph{Kolmogorov quotient} of $X$.
\end{definition}

In words, two points $x$, $y$ belonging to the same equivalent class with respect to $\equiv$ are \emph{topologically indistinguishable}.

\begin{remark}\label{rem: sulla topologia quoziente}
Notice that the topology on $X_{/\equiv}$ is the quotient topology, namely a set $U$ is open in $X_{/\equiv}$ if and only if $\pi^{-1}(U)$ is open in $X$, where $\pi\colon X\to X_{/\equiv} $ is the (natural) projection onto $X_{/\equiv}$.
\end{remark}

\begin{proposition}\label{prop: il limite diretto e il quoziente di Kolmogorov}
Let $\B\in\inj$ which carries a faithful state $s$. Then the Kolmogorov quotient of $\B$ is its Booleanisation $\bool$.
\end{proposition}
\begin{proof}
We have to show that two elements $a,b\in B$ are topologically indistinguishable (namely, $d_{s}(a,b)= 0$) if and only if $a\sim b $. Assume w.l.o.g. that $a\in A_{i}$ and $b\in A_{j}$ with $i\neq j$ and set $k=i\vee j$. \\
\noindent
($\Rightarrow$) Let $d_{s}(a,b)= 0$ (i.e. $a,b$ are indistinguishable). Then $s(a\vartriangle^{\B} b) = 0$ and, since $s$ is faithful, $a\vartriangle^{\B} b = p_{ik}(a) \vartriangle^{\A_{k}} p_{jk}(b) = 0_{k}$. Then $p_{ik}(a) = p_{jk}(b) $, i.e. $a\sim b$.
 \\
\noindent
($\Leftarrow$) Let $a\sim b$. It follows that there exists $l\in I$ such that $i,j\leq l$ and $p_{il}(a)=p_{jl}(b)$. Then
\begin{align*}
s (a\vartriangle b) & = s (p_{ik}(a) \vartriangle^{\A_{k}} p_{jk}(b)) &  \\
& = s_{k} (p_{ik}(a) \vartriangle^{\A_{k}} p_{jk}(b)) \\
& = s_{l} (p_{kl}(p_{ik}(a) \vartriangle^{\A_{k}} p_{jk}(b))) & \text{(Proposition \ref{prop: restrizione su Ai e uno stato})}  \\ 
& = s_{l} (p_{kl}\circ p_{ik}(a) \vartriangle^{\A_{l}} p_{kl}\circ p_{jk}(b)) & \\
& = s_{l} ( p_{il}(a) \vartriangle^{\A_{l}} p_{jl}(b)) \\
& = s (0_{l}) \\
& = 0.
\end{align*}
This shows that $a,b$ are indeed topologically indistinguishable points.
\end{proof}

The proof of Proposition \ref{prop: il limite diretto e il quoziente di Kolmogorov} shows another interesting fact worth being highlighted: while $\B$ is a pseudo-metric space, its Booleanisation $\bool$ becomes a \emph{metric} space. 

By combining Proposition \ref{prop: il limite diretto e il quoziente di Kolmogorov} and Theorem \ref{th: fatti su B e Bool} we immediately get the following. 

\begin{corollary}
Let $\B\in\inj$ with a faithful state $s$. Then $\widehat{\B}_{/_{\equiv}} = \widehat{\A}_{\infty}$, i.e. the Kolmogorov quotient of the completion $(\B,d_{s}) $ is the completion of the metric space $(\bool, d_{\infty})$.
\end{corollary}

Recall that, given a surjective map $f\colon X\to Y$ between topological spaces $X$ and $Y$, a \emph{section} of $f$ is a continuous map $g\colon Y\to X$ such that $f\circ g = id_Y$.

\begin{theorem}\label{th: fatti su B e Bool}
The following facts hold for the topological spaces $\B$ and $\bool$:  
\begin{enumerate}
\item There exists a section $\sigma\colon \bool\to \B$ of $\pi$ such that $\sigma(\bool)$ is dense in $\B$; 
\item $\sigma$ preserves states, namely $s\circ\sigma = \Phi(s)$. 
\end{enumerate}
\end{theorem}

\begin{proof}

(1) Consider $\sigma\colon \bool\to \B$ defined as follows: 
\[
\sigma([a]_{\sim}) = a, 
\]
where $a$ is picked by a choice function among the elements of the equivalence class $[a]_{\sim}$ (this can be done recurring to the Axiom of choice). We first show that $\sigma$ is continuous. Let $D_{r}(b)$ an the open disk of radius $r$ centered in b, for some $b\in B$. $\sigma^{-1}(D_{r}(b)) = \{[y]_{\sim}\in A_{\infty}\mid\sigma([y]_{\sim})\in D_{r}(b) \} = \{[y]_{\sim}\in A_{\infty}\mid y\in D_{r}(b) \} = \{[y]_{\sim}\in A_{\infty}\mid d_{s}(y,b) < r \} = \{[y]_{\sim}\in A_{\infty}\mid d_{\infty}([b]_{\sim},[y]_{\sim}) < r\} = D_{r}([b]_{\sim}) $, the open disk of $\bool$ centered in $[b]_{\sim}$. This shows the continuity of $\sigma $. \\
\noindent
The fact that  $\pi\circ\sigma = id$ immediately follows from the definition of $\sigma$.\\
\noindent
To show that $\sigma(\bool)$ is dense in $\B$, we have to check that $\sigma(\bool)\cap U\neq \emptyset$, for every non-void open set $U\subset B$. Since $U$ is non empty, there exists $b\in U$. Moreover, since $U$ is open, then, for some $r > 0$, $D_{r(b)}\subset U$ (where $D_{r(b)} = \{y\in B \mid d_{s}(b,y) < r\}$). W.l.o.g. we can assume that $b\in A_{i}$, for some $i\in I$. Observe that, for each $j\in I$ with $i\leq j$, we have that $p_{ij}(b)\in U$. Indeed $d_{s}(b,p_{ij}(b)) = s(b\vartriangle^{\B} p_{ij}(b)) = s_{j}(p_{ij}(b)\vartriangle^{\A_{j}} p_{ij}(b)) = 0$, which implies that $p_{ij}(b)\in D_{r(b)}\subset U$. If $\sigma([b]) = b$, then we have finished. So, assume that $\sigma([b]) = a$, with $a\neq b$. W.l.o.g. let $a\in A_{j}$ (for some $j\in I$). By definition of $\sigma$, $a\in [b]_{\sim}$, i.e. there exists some $k\in I$, such that $i,j\leq k$ and $p_{ik}(b) = p_{jk}(a)$. Since $p_{ik}(b)\in U$ (for the above observation), then also $p_{jk}(a)\in U$. Reasoning as above, one checks that $d_{s}(a,p_{jk}(a)) = 0$, which implies that $a\in U$. This shows that $\sigma(\bool)\cap U\neq \emptyset$, for every non-void open set $U\subset B$, i.e. $\sigma(\bool)$ is dense in $\B$. \\
\noindent
(2) follows from the definition of $\sigma$ and Theorem \ref{th: corrispondenza stati ISBL e booleanizzazione}.
\end{proof}

\begin{remark}
Observe that the projection $\pi$ admits many sections (depending on the cardinality of the fiber $\pi^{-1}([b])$, for $b\in B$) and all of them are topological embeddings by construction.
\end{remark}

Recall that, if $f\colon X\to Y$ be a continuous map between two topological spaces ($X$ and $Y$), $V\subseteq X$ is $f$-saturated (or saturated with respect to $f$) if $V=f^{-1}(f(V))$.

\begin{lemma}\label{lemma: saturi rispetto a pi}
Every open and closed set of an involutive bisemilattice $\B$ is saturated with respect to the projection $\pi\colon \B\to\bool$. In particular, $\pi$ is (continuous) open and closed.
\end{lemma}
\begin{proof}
Since the basis of the topology (over $\B$) is the family of open disks, then, with respect to open sets, it is enough to check that $\pi^{-1}(\pi(D_{r})) = D_{r}$, for some open disk $D_r$. Let $D_{r}(b)$ an open disk (of radius $r$) centered in $b$, for some $b\in B$. Then $\pi^{-1}(\pi(D_{r}(b))) = \{ x\in B \mid \pi(x)\in\pi(D_{r}(b)) \} = \{ x\in B\mid d_{\infty}([x],[b]) < r \} = \{ x\in B \mid d_{s}(x,b) < r \} = D_{r}(b) $, where the second last equality holds by Theorem \ref{th: corrispondenza stati ISBL e booleanizzazione}.\\
\noindent
Let $C\subseteq B$ a closed set. Then $C= B \setminus U$, for some open set $U$. Observe that $C\subseteq \pi^{-1}(\pi (C))$ holds in general, so we have to show only the converse inclusion. To this end $\pi^{-1}(\pi (C)) = \pi^{-1}(\pi(B\setminus U)) \subseteq \pi^{-1}(\pi(B)\setminus\pi(U)) = \pi^{-1}(\pi(B))\setminus\pi^{-1}(\pi(U)) = B\setminus U = C$, where the second last equality holds since open sets are saturated with respect to $\pi$. Finally, the fact that $\pi$ is open and closed follows from what we have just proved, observing that $\bool$ is the Kolmogorov quotient of $\B$ (Proposition \ref{prop: il limite diretto e il quoziente di Kolmogorov}), topologised with the quotient topology (see Remark \ref{rem: sulla topologia quoziente}).
\end{proof}

\begin{remark}\label{rem: fatto topologici}
In the proof of the following results we will use some well-known facts in general topology that we briefly recap (see, for instance, \cite{Munkres}). Let $f\colon X\to Y$, be an open and closed continuous function between topological spaces. Then
\begin{enumerate}
\item $f^{-1}(\mathrm{Int}(B)) = \mathrm{Int}(f^{-1}(B))$, for every $B\subseteq Y$;
\item $f(\overline{A}) = \overline{f(A)}$, for every $A\subseteq X$.  
 \end{enumerate}
\end{remark}

\begin{lemma}\label{lemma: chiusi e interno in B}
Let $C\subseteq B$ a closed set of an involutive bisemilattice $\B$. Then $\pi(\mathrm{Int}(C)) = \mathrm{Int}(\pi(C))$.
\end{lemma}
\begin{proof}
Let $C$ be a closed set in $B$. Observe that, from Lemma \ref{lemma: saturi rispetto a pi}, we have that $\pi$ is an open and closed continuous map. Hence 
$$ \pi(\mathrm{Int}(C)) = \pi(\mathrm{Int}(\pi^{-1}(\pi(C)))) = \pi(\pi^{-1}(\mathrm{Int}(\pi(C)))) = \mathrm{Int}(\pi(C)),$$
where we have applied Lemma \ref{lemma: saturi rispetto a pi} and the properties of open and closed continuous maps (see Remark \ref{rem: fatto topologici}). 
\end{proof}

Observe that the statement of Lemma \ref{lemma: chiusi e interno in B} does not hold for all topological spaces (see \cite{Memp} for details). Recall that a map $f\colon X\to Y$ (between two topological spaces) \emph{preserves the interiors} if $\mathrm{Int}(f(A)) = f(\mathrm{Int}(A))$, for all $A\subseteq X$. Interior preserving maps are studied in \cite{Memp}. One can wonder whether the statement of Lemma \ref{lemma: chiusi e interno in B} could be extended to any subset of involutive bisemilattice (not only for closed subsets). Interestingly enough, the next result shows that the projection $\pi$ is interior preserving if and only if $\B$ is a Boolean algebra.

\begin{theorem}\label{th: pi preserva l'interiore}
Let $\B$ an involutive bisemilattice carrying a faithful state. The following facts are equivalent: 
\begin{enumerate}
\item $\B=\bool$; 
\item $\pi\colon\B\to\bool$ is an interior preserving map; 
\item $\sigma(\bool)$ is open (closed, saturated) in $\B$, for every section $\sigma\colon\bool\to \B$.
\end{enumerate}
\end{theorem}
\begin{proof}
(1) $\Rightarrow$ (2) is trivial (as $\pi = id$). \\
(2) $\Rightarrow$ (1). We reason by contraposition, and suppose that $\B\neq \bool$. This implies that there exists an element $[a]\in A_{\infty}$ such that $\mid \pi^{-1}([a])\mid \geq 2 $. Let $b\in\pi^{-1}[a]$. Observe that $\bool\setminus\{[a]\}$ is open (as $\{[a]\}$ is closed), thus, since $\pi$ is continuous, $\B\setminus\{\pi^{-1}([a])\}$ is open. This implies that $\mathrm{Int}(\B\setminus\{b\}) = \B\setminus\{\pi^{-1}([a])\}$. Then, $\pi(\mathrm{Int}(\B\setminus\{b\})) = \pi(\B\setminus\{\pi^{-1}([a])\})= \bool\setminus\{[a]\}$. On the other hand, $\mathrm{Int}(\pi(\B\setminus\{b\})) = \bool$, since $\mid \pi^{-1}([a])\mid \geq 2 $, which shows that $\pi$ does not preserve interiors. \\
\noindent
(1) $\Rightarrow$ (3) is obvious. \\
\noindent
(3) $\Rightarrow$ (1). Let $\sigma(\bool)$ be open (closed, saturated) in $B$. Then, by Lemma \ref{lemma: saturi rispetto a pi}, $\sigma(\bool)$ is $\pi$-saturated, i.e. $\sigma(\bool)= \pi^{-1}(\pi(\sigma(\bool))) = \pi^{-1}(\bool) = \B $, so $\pi$ is a bijection being $\sigma$ its inverse.
\end{proof}


Recall that, for a topological space $X$, an open set $U\subseteq X$ is an \emph{open regular set} if $U=\mathrm{Int}(\overline{U})$ (where $\overline{U}$ indicates the closure of $U$). To keep in mind the difference between an open, and an open regular set, consider $\mathbb{R}$ topologised (as usual) with the Euclidian topology. Then $(0,1)$ is an example of an open regular set, while $U = (0,1)\cup (1,2)$ is an open set which is not regular, as $\mathrm{Int}(\overline{U}) = (0,2) $. The set of open regular sets $Reg(X)$ of a topological space $X$ can be turned into a (complete) Boolean algebra (see, for instance, \cite{GivantHalmos}) $\mathbf{Reg}(X)=\langle Reg(X), \cap, \vee, \smallsetminus,\emptyset, X \rangle$, where $U\vee V\coloneqq \mathrm{Int}(\overline{A\cup B})$. Moreover, the Boolean algebra of $\mathbf{Clopen}(X)$ (of the clopen sets of $X$) is a subalgebra of $\mathbf{Reg}(X)$. Despite the fact that an involutive bisemilattice $\B$ and its Booleanisation $\bool$ are not homeomorphic (except in the trivial case $\B=\bool$), surprisingly enough, the Boolean algebras of regular sets arising from $\B$ and $\bool$ are isomorphic, as shown in the following. 
 
\begin{theorem}\label{th: corrispondenza aperti IBSL e Booleanizzazione}
Let $\B\in\inj$ carrying a faithful state. The projection $\pi\colon\B\to\bool$ induces a bijection between $\mathsf{Open}(\B)$ and $\mathsf{Open}(\bool)$, the open sets of $\B$ and $\bool$, respectively. \\
Moreover, the Boolean algebras $\mathbf{Reg}(B)$ and $\mathbf{Reg}(A_{\infty})$ are isomorphic.
 \end{theorem}
 \begin{proof}
The fact that $\pi$ is a bijection between $\mathsf{Open}(\B)$ and $\mathsf{Open}(\bool)$ follows from Lemma \ref{lemma: saturi rispetto a pi} (and the surjectivity of $\pi$). 
The isomorphism between $\mathbf{Reg}(B)$ and $\mathbf{Reg}(A_{\infty})$ is given by the projection $\pi$, restricted to $\mathbf{Reg}(B)$. We first show that the map is well defined, i.e. that given an open regular set $U\in \mathbf{Reg}(B)$, then $\pi(U)\in\mathrm{Reg}(\bool)$. To show regularity, observe that 
\begin{align*}
\pi(U) & = \pi(\mathrm{Int}(\overline{U})) & (U \text{ is regular})  \\
& = \mathrm{Int}(\pi(\overline{U})) & (\text{Lemma \ref{lemma: chiusi e interno in B}}) \\
& = \mathrm{Int}(\overline{\pi(U)}) & (\pi \text{ is continuous, open and closed})
\end{align*}
To conclude the proof, we only need to check that $\pi$ is a homomorphism (with respect to the Boolean operations of $\mathbf{Reg}(B)$ and $\mathbf{Reg}(A_{\infty})$). With respect to the constants, observe that $\pi(\emptyset)= \emptyset$ and, since $\pi$ is surjective, $\pi(B) = A_{\infty}$. Now, let $U,V\in\mathrm{Reg}(B)$, then $\pi(U)\cap\pi(V) = \pi\circ\pi^{-1}(\pi(U)\cap\pi(V)) = \pi(\pi^{-1}(\pi(U))\cap \pi^{-1}(\pi(V))) = \pi(U\cap V)$, where the last equality follows from Lemma \ref{lemma: saturi rispetto a pi}. 
Moreover, 
\begin{align*}
\pi(U\vee V) & = \pi(\mathrm{Int}(\overline{U\cup V})) &   \\
& = \mathrm{Int}(\pi(\overline{U\cup V})) & (\text{Lemma \ref{lemma: chiusi e interno in B}}) \\
& = \mathrm{Int}(\overline{\pi(U\cup V)}) & (\pi \text{ is continuous, open and closed}) \\
& = \mathrm{Int}(\overline{\pi(U)\cup \pi(V)}) & \\
& = \pi(U)\vee\pi(V).
\end{align*}
Since we have shown that $\pi$ preserves the constants and the binary operations, it follows that it preserves also the unary operation $\setminus$, hence we are done. 
 \end{proof}

\begin{theorem}[Topological characterization of states]
Let $s$ be a faithful state over $\B$ and $t\colon \B\to [0,1]$ a continuous map such that $t\circ\sigma = \Phi(s)$, for any section $\sigma\colon\bool\to \B$. Then $t = s$.
\end{theorem}
\begin{proof}
By assumption, $t\circ\sigma = \Phi(s)$, for any section $\sigma\colon\bool\to \B$. This implies that the two maps $s$ and $t$ coincide over a dense subset $\sigma(\bool)$ of $\B$ (in virtue of Theorem \ref{th: fatti su B e Bool}-(2)). Therefore, since both $s$ and $t$ are continuous ($s$ by Remark \ref{rem: lo stato e sempre una funzione continua}, $t$ by assumption) and $[0,1]$ is Hausdorff, we have $t =s$. 
\end{proof}

\begin{remark}
As we have seen, in general, the spaces $\B$ and $\bool$ are not homeomorphic. However, it follows from the general theory of Kolmogorov quotients (see \cite[Theorem 8.6]{finlandese}) that, under the assumption that they are both Alexandrov discrete\footnote{A topological space $X$ is \emph{Alexandrov discrete} when the arbitrary intersection of open sets is an open set.} they are homotopically equivalent. Obviously, the equivalence holds in the particular case whether $\B$ is finite.
\end{remark}

%
%
 

%


\section{Conclusion and further work}\label{sec: conclusioni}

In this work, we have shown how to define a notion of state on P\l onka sums of Boolean algebras, with the aim of expressing the probability for elements of an involutive bisemilattice, a variety associated to the logic PWK. In particular, we have shown that the (non-trivial) elements of the class $\N$ (of involutive bisemilattices with no trivial algebra in the P\l onka sum representation) always carry a state. The class $\N$ plays a relevant role in logic, as algebraic counterpart of the extension of PWK by adding the ex-falso quodlibet. Moreover, we have exploited the connections between such notion, the probability measures carried by Boolean algebras in a P\l onka sum and the Booleanisation of an involutive bisemilattice. These connections are crucial in the study of the completion and the topology induced by a state over an involutive bisemilattice. 

This work sheds a further light on the possibility of developing the theory of probability beyond the boundaries of \emph{classical} events, namely elements of a Boolean algebra. To the best of our knowledge, this consists of the first attempt to lift (finitely additive) probability measures from Boolean algebras to P\l onka sums of Boolean algebras. For this reason, many theoretical problems, as well as potential applications are not examined in the present work. At first, it shall be noticed that there is nothing special behind the choice of Boolean algebras, a part the fact that P\l onka sums of Boolean algebra play the important role of characterising the algebraic counterparts of Paraconsistent Weak Kleene logic and its extensions. The ideas developed here could be used, in principle, to define states for varieties that are represented as P\l onka sums of classes of algebras admitting states, such as MV-algebras, Goedel algebras, Heyting algebras, just to mention some for which a theory of states has been developed. On the other hand, a deeper investigation about the connection between P\l onka sums and certain logics has been conducted in \cite{BonzioMorascoPraBaldi} and \cite{BonzioPraBaldi}. 

A relevant question that we leave for further investigations is the possibility of characterising states over involutive bisemilattices as \emph{coherent books} over a (finite) set of events of the extension of the logic PWK. Coherent books have been introduced, in the classical case, by de Finetti \cite{DeFinetti31,DeFinetti74}, via a specific (reversible) betting game and are shown to be in one-to-one correspondence with (finitely additive) probability measures over the Boolean algebra generated by the events considered. This kind of abstract betting scenario has been used also to characterise states for non-classical structures \cite{MUNDICIbookmaking}.

We have shown (see Theorem \ref{th: rappresentazione integrale stati IBSL}) that states over involutive bisemilattices correspond to integrals on the dual space of the Booleanisation. It makes sense to ask whether this correspondence can be extended to faithful states, relying on the integral representation proved for faithful states over free MV-algebras in \cite{flaminio2019}. 

The theory of states we developed could, perhaps, find potential applications also in the field of \emph{knowledge representation}. This is mainly due to the fact that states break into probability measures over the Boolean algebras in the P\l onka sum representation. One may interpret the semilattice of indexes, involved in the representation, to model, for instance, situations of branching time\footnote{A similar idea is developed from the construction of horizontal sums in \cite{temporalBL}.} (as the index set is, in general, not a chain). A state, then, encapsulates information related to the probabilities of classical events (Boolean algebras) located in every point (indexes) of the structure. This might be used, in principle, also to analyse conditional bettings or counterfactual situations, under the assumption, for instance, that events are related when there is a homomorphism connecting the algebras they belong to.   


\section*{Appendix A}

Our definition of state relies (see Definition \ref{def: stato su IBSL}) on the assumption that two elements $a,b\in B$ of an involutive bisemilattice $\B$ are logically incompatible provided that $a\wedge^{\B} b = 0_{i}$, where $0_{i}$ is the bottom element of the Boolean algebra (in the P\l onka sum representation of $\B$) where the operation $\wedge$ is computed. One could question this principle and understand two elements $a,b\in B$ as incompatible, in case $a\wedge^{\B} b = 0$.
This leads to a different definition of state obtained, by replacing condition (2) in Definition \ref{def: stato su IBSL} with the following:
 
\begin{equation}\label{eq: definizione alternativa stato}
s(a\vee b) = s(a) + s(b)  \text{ provided that  } a\wedge b = 0. 
\end{equation}

However, since the element $0$ of an involutive bisemilattice always belongs to the Boolean algebra (in the P\l onka sum) whose index is the least element in the semilattice $\langle I,\leq\rangle$ of indexes, this latter choice leads to the following consequence.

\begin{proposition}\label{prop: stati nel caso N0=0}
Let $\B$ and involutive bisemilattice. Then the following are equivalent: 
\begin{enumerate}
\item $s\colon\B\to[0,1]$ satisfies $s(1) = 1$ and condition \eqref{eq: definizione alternativa stato}; 
\item $s_{i_0}$ is a (finitely additive) probability measure over the Boolean algebra $\A_{i_0}$ where $i_0$ is the minimum element in $I$.
\end{enumerate}
\end{proposition}
\begin{proof}
(1) $\Rightarrow$ (2). Immediate by observing that $1\in A_{i_0}$ and that, for any arbitrary pair of elements $a,b\in B$, $a\land b =0$ implies that $a,b\in A_{i_0}$.\\

\noindent
(2) $\Rightarrow$ (1). Let $s_{i_{0}}\colon \A_{i_0}\to[0,1]$ be any finitely additive probability measure over $\A_{i_{0}}$. Then, the map $s\colon\B\to[0,1]$
\[
s(x)\coloneqq\begin{cases}
s_{i_{0}}(x) & \text{if $x\in A_{i_0}$,} \\
\alpha & \text{otherwise,}
\end{cases}
\]
for $\alpha\in (0,1)$ a fixed number, satisfies that $s(1) = 1$ and that $s(a\vee b) = s(a) + s(b)$, when $a\wedge b = 0$. Moreover, $s_{i_{0}}$ is the restriction of $s$ over $\A_{i_0}$.
\end{proof}

In words, the above result suggests that, this different notion of state, obtained by replacing (2) in Definition \ref{def: stato su IBSL} with \eqref{eq: definizione alternativa stato}, implies that only the elements belonging to the Boolean algebra $\A_{i_{0}}$ are actually measured following the standard rules of probability. 

\section*{Appendix B}

We have shown that (faithful) states on involutive bisemilattices are in correspondence with (regular) probability measures over the corresponding Booleanisations (see Theorem \ref{th: corrispondenza stati ISBL e booleanizzazione} and Theorem \ref{th: corrispondenza stati ISBL e booleanizzazione (caso faithful)}). However, there could be many non-isomorphic (injective) involutive bisemilattices having the same Booleanisation.\footnote{It is clear that two isomorphic involutive bisemilattices have isomorphic Booleanisations.} The problem of characterising all injective involutive bisemilattices having the same Booleanisation (outlined at the end of Section \ref{Sec: faithful states}) is not an easy task and the aim of present appendix is (to try) to provide some reasons.

At first, we observe that a knowledge of the lattice of subalgebras of Boolean algebras (see \cite{Gra72,Rao1979}) is not sufficient to give an answer to the problem, even in case the index set $I$ is finite and, consequently, the Booleanisation coincides with the Boolean algebra (in the P\l onka sum) whose index is the top element of the lattice of indexes. Indeed the structure of an injective involutive bisemilattices strongly relies on all possible embeddings one can have between algebras in the sum. To give a more precise intuition, consider the involutive bisemilattice introduced in Example \ref{ex: operazioni Plonka} and a different one, constructed recurring to the same index set and the same algebras, but with different embeddings, namely $p_{ik}(a) = d = p_{jk}(b)$. The two involutive bisemilattices have the same Booleanisation $\A_{k}$, although they are not isomorphic; this can be checked directly, or also reasoning the categorical equivalence between involutive bisemilattices and semilattice direct systems of Boolean algebras shown in \cite{Loi} (see also \cite{SB18}). 

Given a Boolean algebra $\bool$, an obvious example of (finite) injective involutive bisemilattice is one constructed over a direct system formed by subalgebras of $\bool$ (with $\bool$ as top element of the index set) and whose homomorphisms are inclusions. We will call these kinds of involutive bisemilattices \emph{inclusive} (see Definition below). We wonder whether it is possible to count the number of non-isomorphic \emph{inclusive} involutive bisemilattices having the same Booleanisation. We will show some peculiar cases, which, in our view, give a gist of the hardness of the problem announced in Section \ref{Sec: faithful states} (which we cannot solve here).\footnote{A similar problem, connecting with counting the number of a specific subclass of involutive bisemilattices has been addressed in \cite{BonzioValota}.}\\

In order to define inclusive involutive bisemilattices, let $\bool$ a finite Boolean algebra such that $\mid A_{\infty}\mid = 2^{n}$ and consider a collection $\{\A_{1},\dots,\A_{n}\}$ of distinct subalgebras of $\bool$ such that $\A_{1} = \A_{i_0}$, $\A_{n}=\bool$ and, for each $i,j\in\{1,\dots, n\}$, either $A_{i}\subset A_{j}$ or $A_{j}\subset A_{i}$. In other words, the collection of subalgebras $\{\A_{1},\dots,\A_{n}\}$ consists of a maximal subchain in the lattice of all subalgebras of $\bool$.\footnote{By \emph{maximal subchain} we  mean a sublattice which is a chain and contains a copy of any Boolean subalgebra (of $\bool$) with cardinality less or equal to $2^{n}$. }

\begin{definition}\label{def: IBSL inclusivi}
A finite involutive bisemilattice $\B$ with Booleanisation $\bool$ is \emph{inclusive} if it is the P\l onka sum over a direct system whose elements are subalgebras in the collection $\{\A_{1},\dots,\A_{n}\}$ and homomorphisms are inclusions from $A_{l}$ to $A_{m}$ in case $l\leq m$, for some $l,m\in \{1,\dots, n\}$.
\end{definition}
We will call \emph{weight} $k$ of an involutive bisemilattice $\B$, the number of algebras in its P\l onka sum representation.
Observe that all inclusive involutive bisemilattice are injective and that, since, we are considering finite algebras then the index set forms a lattice. \\

\textbf{Problem:} compute the number $N(\bool, k)$ of all non-isomorphic inclusive involutive bisemilattices of weight $k$ and Booleanisation $\bool$.\\

Expecting to find a formula counting $N(\bool, k)$ is a hopeless effort. Indeed, counting the number of non-isomorphic inclusive involutive bisemilattice constructed over a direct system containing $k$ copies of the same algebra ($\bool$), is equivalent to counting the number of all non-isomorphic finite lattices with $k$ elements. This is a problem for which no formula is known to work and which is indeed solved by a specific algorithm \cite{HJ02} (implemented and improved in the case of modular lattices in \cite{Jipsencounting}). 
However, there are fortunate cases, where it is possible to find a formula counting $N(\bool, k)$. We address the case where $\B$ is an inclusive involutive bisemilattice such that $B\smallsetminus \{A_{1},A_{\infty}\}$ consists of $k-2$ distinct subalgebras of $\bool$ (thus, $k-2\leq n$). We refer to the number of all non-isomorphic inclusive involutive bisemilattices, in this first case, as $N_{d}(\bool, k)$.\\
\noindent
To make an example of the case under consideration here, let $\bool$ be the eight-elements Boolean algebra ($n=3$). Then the collection $\{\A_{1},\A_{2}, \A_{3}\}$ consists of a copy of the two-elements Boolean algebra ($\A_1)$, a copy of the four-elements Boolean algebra ($\A_{2}$) and the eight-elements ($\A_{3}$) one. It can be immediately checked that, for instance, for $k=4$, we have $N_{d}(\bool, k) = 8$. The P\l onka sum representation of the eight non-isomorphic involutive bisemilattices are depicted in the following drawing (where arrows stands for inclusions).

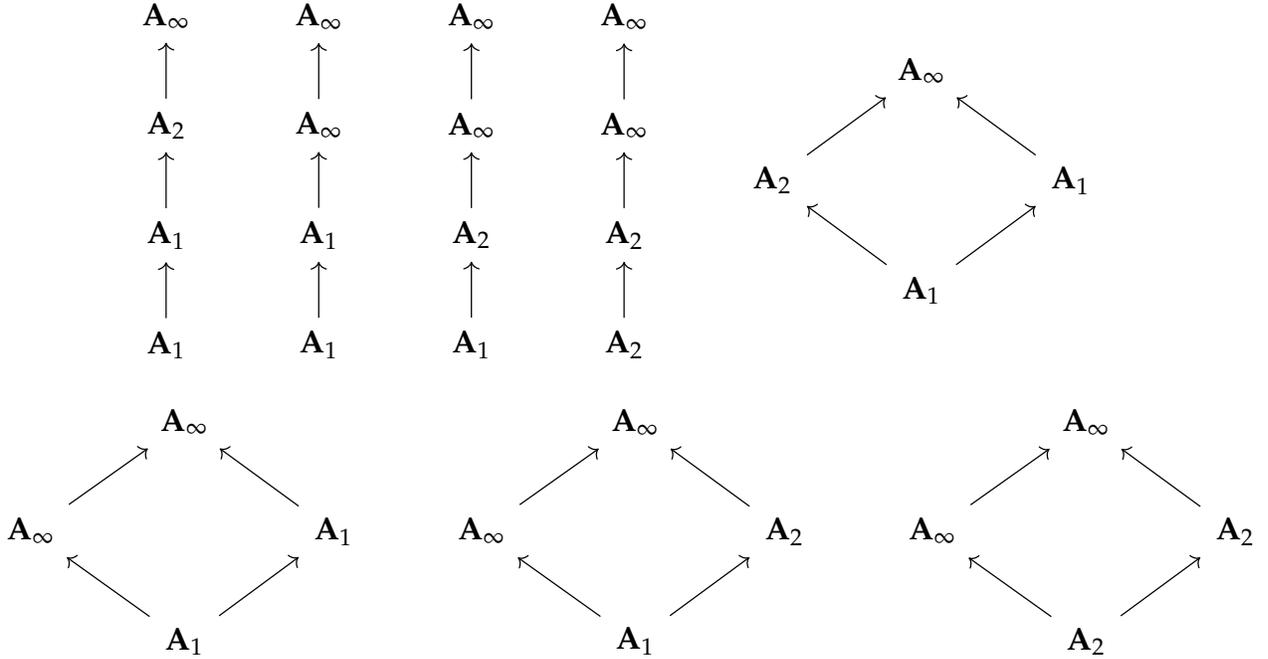
\[
 \begin{tikzcd}
 \mathbf{A}_{\infty}  \\
  \mathbf{A}_{2}\arrow[u] \\
   \mathbf{A}_{1}\arrow[u] \\
  \mathbf{A}_{1}\arrow[u] 
 \end{tikzcd}
  \hspace{1cm}
  \begin{tikzcd}
 \mathbf{A}_{\infty}  \\
  \mathbf{A}_{\infty}\arrow[u] \\
   \mathbf{A}_{1}\arrow[u] \\
  \mathbf{A}_{1}\arrow[u] 
 \end{tikzcd}
 \hspace{1cm}
  \begin{tikzcd}
 \mathbf{A}_{\infty}  \\
  \mathbf{A}_{\infty}\arrow[u] \\
   \mathbf{A}_{2}\arrow[u] \\
  \mathbf{A}_{1}\arrow[u] 
 \end{tikzcd}
 \hspace{1cm}
  \begin{tikzcd}
 \mathbf{A}_{\infty}  \\
  \mathbf{A}_{\infty}\arrow[u] \\
   \mathbf{A}_{2}\arrow[u] \\
  \mathbf{A}_{2}\arrow[u] 
 \end{tikzcd}
  \hspace{1cm}
   \begin{tikzcd}
 & \mathbf{A}_{\infty} &  \\
  \mathbf{A}_{2}\arrow[ur] & &   \mathbf{A}_{1}\arrow[ul] \\
&  \mathbf{A}_{1}\arrow[ul]\arrow[ur] &
 \end{tikzcd}
\]

\[
 \begin{tikzcd}
 & \mathbf{A}_{\infty} &  \\
  \mathbf{A}_{\infty}\arrow[ur] & &   \mathbf{A}_{1}\arrow[ul] \\
&  \mathbf{A}_{1}\arrow[ul]\arrow[ur] &
 \end{tikzcd}
 \hspace{1cm}
 \begin{tikzcd}
 & \mathbf{A}_{\infty} &  \\
  \mathbf{A}_{\infty}\arrow[ur] & &   \mathbf{A}_{2}\arrow[ul] \\
&  \mathbf{A}_{1}\arrow[ul]\arrow[ur] &
 \end{tikzcd}
 \hspace{1cm}
 \begin{tikzcd}
 & \mathbf{A}_{\infty} &  \\
  \mathbf{A}_{\infty}\arrow[ur] & &   \mathbf{A}_{2}\arrow[ul] \\
&  \mathbf{A}_{2}\arrow[ul]\arrow[ur] &
 \end{tikzcd}
\]

Fortunately, in this case, it is not necessary to recur to the algorithm counting all (non-isomorphic) finite lattices.
To see why, we begin by providing an useful lemma, which allows us to count the number of (non-isomorphic) inclusive involutive bisemilattice which differs only for the Boolean algebra whose index is the least element in the index lattice. To this end, let $X=\{1,\dots, n\}$ consisting of the first $n$ natural numbers. Let $1\leq h\leq n$ be fixed, and, for each $s\in X$, let $\mathcal{P}_{_s}(h)$ be the set consisting of all subsets of $X$ of cardinality $h$, having $s$ as minimum element (with respect to the natural linear order over $X$). 
\begin{lemma}\label{lem: lemmino ausiliario}
$\displaystyle\sum_{s=1}^{n}s\mid\mathcal{P}_{_s}(h)\mid = \binom{n+1}{h+1}$.
\end{lemma}
\begin{proof}
Observe that there are $n-s$ possible choices of subsets of $X$ of $h$ elements containing $s$, thus $\mid\mathcal{P}_{_s}(h)\mid =\displaystyle \binom{n-s}{h-1} $. Therefore $\displaystyle\sum_{s=1}^{n}s\mid\mathcal{P}_{_s}(h)\mid = \displaystyle\sum_{s=1}^{n}s \binom{n-s}{h-1}= \binom{n+1}{h+1}$, where the last equality easily follows by induction over $n$ (for fixed $h$).
\end{proof}

A simple example may help to grasp the content of the previous Lemma and its utility for our purposes. Let $n=3$, so $X=\{1,2,3\}$, and $h=2$. Then, $\mathcal{P}_{_1}(2) =\{ \{1,2\}, \{1,3\}\}$ and $\mathcal{P}_{_2}(2) =\{ \{2,3\}\}$. Thus, $\displaystyle\sum_{s=1}^{n}s\mid\mathcal{P}_{_s}(h)\mid = 1\cdot 2 + 2\cdot 1 = 4 = \binom{3+1}{2+1} = \binom{4}{3}$. 

\begin{theorem}\label{th: N_k caso distinti}
Let $\B$ be an inclusive involutive bisemilattice with Booleanisation $\bool$, whose P\l onka sum representation consists of $k-2$ distinct subalgebras of $\bool$. Then \\
\noindent
$N_{d}(\bool, k)= \displaystyle\binom{n+1}{k-1}a(k-2) $, where $a(k-2)$ is the number of acyclic graph with $k-2$ vertices.
\end{theorem}
\begin{proof}
Observe that all the possible different (non isomorphic) inclusive involutive bisemilattices obtained by setting a different algebra in the place of the least element in the index set are counted, via Lemma \ref{lem: lemmino ausiliario} (setting $h=k-2$), by $\displaystyle \binom{n+1}{k-1} $. Finally, we have to count the possible posets with $k-2$ elements and this number is equivalent to the number $a(k-2)$ is the number of acyclic graph with $k-2$ vertices. This is counted by the following inductive formula:
\begin{equation*}
a(k-2) = \sum_{q=1}^{k-2}\binom{k-3}{q-1}q^{(q-2)}a(k-2-q),
\end{equation*}
where $q^{(q-2)}$ is the Cayley's formula, counting the number of trees over $q$ vertices (of the acyclic graph), where, by convention, we assume $a(0)=1$.
\end{proof}

Solving the general case, namely determining a way to count $N(\A_{\infty},k)$ could be delivered by ``combining'' two relevant cases: the one where the involutive bisemilattice is constructed via $k-2$ distinct subalgebras of the Booleanisation $\bool$ and the relevant one where it contains an arbitrary number of copies of the same subalgebra of $\bool$. As mentioned, the combination of the two cases gives itself raise to a hard problem which requires the application of an algorithm for counting the number of non-isomorphic (finite) lattices. Yet, an understanding of ``how many'' those algebras are (the number of injective involutive bisemilattices is greater that the inclusive ones!) does not provide a criterion to characterise when two injective involutive bisemilattices have the same Booleanisation (our original problem), but gives a hint of the difficulty of the enterprise, which we were not able to solve in the present work.

\section*{Acknowledgments}

The first author gratefully acknowledges the financial support of the PRIN project ``From models to decisions'' (Italian Ministry of Scientific Research grant n. 201743F9YE, Turin unit) and also the support of the Marie Curie fellowship within the program ``Beatriu de Pinos'', co-funded by Generalitat de Catalunya and the European Union's Horizon 2020 research and innovation programme under the MSCA grant agreement No. 801370. The second  author was  supported by PRIN 2015 -- Real and Complex Manifolds; Geometry, Topology and Harmonic Analysis -- Italy, by INdAM, GNSAGA - Gruppo Nazionale per le Strutture Algebriche, Geometriche e le loro Applicazioni, and by KASBA - Funded by Regione Autonoma della Sardegna. We finally thank Roberto Giuntini, Francesco Paoli, Tommaso Flaminio and two anonymous reviewers for their fruitful comments on a previous version of the paper.


\end{document}